\documentclass[11pt]{amsart}
\usepackage{amsmath,amscd,amssymb}
\usepackage[mathscr]{eucal}
\usepackage[all]{xy}

\newcommand{\Wip}{\mathrm{A}_+^1}

\newcommand{\vanish}[1]{\relax}

\newcommand{\N}{\mathbb{N}}

\newcommand{\R}{\mathbb{R}}
\newcommand{\C}{\mathbb{C}}

\newcommand{\ud}{\mathrm{d}}

\newcommand{\Sum}[2][\relax]{%
 \ifx#1\relax \sideset{}{_{#2}}\sum
 \else \sideset{}{^{#1}_{#2}}\sum
 \fi}

\DeclareMathOperator{\Sect}{Sect}

\DeclareMathOperator{\dom}{dom}
\DeclareMathOperator{\ran}{ran}

\newcommand{\cls}[1]{\overline{#1}}

\newcommand{\norm}[2][\relax]{%
   \ifx#1\relax \ensuremath{\left\Vert#2\right\Vert}
   \else \ensuremath{\left\Vert#2\right\Vert_{#1}}
   \fi}

\makeatletter
\newcommand{\sprod}[2]{\ensuremath{%
  \setbox0=\hbox{\ensuremath{#2}}
  \dimen@\ht0
  \advance\dimen@ by \dp0
  \left(\left.#1\rule[-\dp0]{0pt}{\dimen@}\,\right|#2\hspace{1pt}\right)}}
 \makeatother

\newcounter{aufzi}

\newcounter{aufzii}

\newcounter{aufziii}

 \newtheorem{thm}{Theorem}[section]
 \newtheorem{cor}[thm]{Corollary}
 \newtheorem{lemma}[thm]{Lemma}
 \newtheorem{prop}[thm]{Proposition}

 \theoremstyle{definition}
 \newtheorem{defn}[thm]{Definition}

 \theoremstyle{remark}
 \newtheorem{rem}[thm]{Remark}

\newtheorem{example}[thm]{Example}

\newtheorem{remark}[thm]{Remark}

\numberwithin{equation}{section}

\begin{document}
\title[On subordination of holomorphic semigroups]{On subordination of holomorphic semigroups}

\author{Alexander Gomilko}
\address{Faculty of Mathematics and Computer Science\\
Nicolas Copernicus University\\
ul. Chopina 12/18\\
87-100 Toru\'n, Poland }

\email{gomilko@mat.umk.pl}

\author{Yuri Tomilov}
\address{Institute of Mathematics\\
Polish Academy of Sciences\\
\'Sniadeckich 8\\
00-956 Warszawa, Poland\\
and Faculty of Mathematics and Computer Science\\
Nicolas Copernicus University\\
ul. Chopina 12/18\\
87-100 Toru\'n, Poland
 }

\email{tomilov@mat.umk.pl}

\thanks{This work was  partially supported by the NCN grant
 DEC-2011/03/B/ST1/00407 and by the EU grant  ``AOS'', FP7-PEOPLE-2012-IRSES, No 318910.}



\date{\today}

\begin{abstract}
We prove that  for any Bernstein function $\psi$ the operator
$-\psi(A)$ generates a  holomorphic
$C_0$-semigroup $(e^{-t\psi(A)})_{t \ge 0}$ on a Banach space,
whenever $-A$ does. This answers a question posed by Kishimoto and
Robinson. Moreover, giving a positive answer to a question by
Berg, Boyadzhiev and de Laubenfels,
 we show that $(e^{-t\psi(A)})_{t \ge 0}$ is holomorphic  in the  holomorphy sector of $(e^{-tA})_{t \ge 0},$
and if $(e^{-tA})_{t \ge 0}$ is sectorially bounded in this sector
then $(e^{-t\psi(A)})_{t \ge 0}$ has the same property. We also
obtain new sufficient conditions on $\psi$ in order that, for
every Banach space $X$, the semigroup $(e^{-t\psi(A)})_{t\ge 0}$
on $X$ is holomorphic whenever $(e^{-tA})_{t\ge 0}$ is a bounded
$C_0$-semigroup on $X$. These conditions improve and generalize
well-known results by Carasso-Kato and Fujita.
\end{abstract}

\maketitle
\section{Introduction}

The present paper concerns operator-theoretic and
function-theoretic properties of Bernstein functions and solves
several notable problems which have been left open for some time.

Bernstein functions play a prominent role in probability theory
and operator theory. One of their characterizations, also
important for our purposes, says that a function
$\psi:[0,\infty)\to (0,\infty)$ is Bernstein if and only if there
exists a vaguely continuous semigroup of subprobability measures
$(\mu_t)_{t \ge 0}$ on $[0,\infty)$
  such that
\begin{equation}\label{SubordIntro}
e^{-t\psi(z)}= \int_0^\infty e^{-z s}\,\mu_t(ds), \qquad z\ge 0,
\end{equation}
for all $t \ge 0.$

Let now  $(e^{-tA})_{t \ge 0}$  be a $C_0$-semigroup on a Banach
space $X$ with generator $-A.$ The relation \eqref{SubordIntro}
suggests a way  to define a new bounded $C_0$-semigroup
$(e^{-tB})_{t \ge 0}$ on $X$ in terms of  $(e^{-tA})_{t \ge 0}$
and a Bernstein function $\psi$   as
\begin{equation}\label{subordvec}
e^{-tB}=\int_{0}^{\infty}e^{-sA}\,\mu_t(ds),
\end{equation}
where $(\mu_t)_{t \ge 0}$ is a semigroup of measures given by
\eqref{SubordIntro}. Following \eqref{SubordIntro}, it is natural
to define $\psi(A):=B.$ As it will be revealed in Subsection
\ref{HPfun} below, such a definition of $\psi(A)$ goes far beyond
 formal notation and it respects some rules for operator
functions called functional calculus.

The semigroup $(e^{-t\psi(A)})_{t \ge 0}$ is  subordinated to the
semigroup $(e^{-tA})_{t \ge 0}$ via a subordinator $(\mu_t)_{t \ge
0}.$ The basics of subordination theory was set up by Bochner
\cite{Bo} and Phillips \cite{Phil52}. This approach to
constructing semigroups is motivated by probabilistic
applications, e.g. by the study of L\'evy processes,  but it has
also significant value for PDEs as well.
 As a textbook
example one may mention a classical result of Yosida expressing
$(e^{-tA^\alpha})_{t \ge 0}, \alpha \in (0,1),$ in terms of
$(e^{-tA})_{t \ge 0}$ as in \eqref{subordvec}. The essential
feature of this example is that $C_0$-semigroups
$(e^{-tA^\alpha})_{t \ge 0}$ turn out to be necessarily
holomorphic. This fact stimulated further research on relations
 between functional calculi and Bernstein functions.
Some of them are described below.

An easy consequence of \eqref{subordvec} is that  for a fixed
Bernstein function $\psi$ the mapping
\begin{equation}\label{map}
\mathcal M: -A \mapsto -\psi (A)
\end{equation}
preserves the class of generators of bounded $C_0$-semigroups, and
it is natural to ask whether there are any other important classes
of semigroup generators stable under $\mathcal M.$ In particular,
whether $\mathcal M$ preserves the class of holomorphic
$C_0$-semigroups. The question was originally asked by Kishimoto
and Robinson in \cite[p. 63, Remark]{Rob}. It appeared to be quite
difficult and there have been very few  general results in this
direction  so far.

First, one should probably recall an old and related result due to
Hirsch \cite{HirschInt} saying that $\mathcal M$ maps sectorial
operators into sectorial operators if $\psi$ is a {\it complete}
Bernstein function. (In this case, the definition of $\psi(A)$
relies on certain integral representations involving resolvents,
see Subsection \ref{funHi}). Note that Hirsch's argument does not
give any control over the angles of sectoriality.

A partial answer to the Kishimoto-Robinson question was obtained
in \cite{Laub} where the question was formulated in another form:
whether $\mathcal M$ preserves the class of sectorially bounded
holomorphic $C_0$-semigroups ? It was proved in \cite[Theorem
7.2]{Laub} that for any Bernstein function $\psi$ the operator
$-\psi(A)$ generates a sectorially bounded holomorphic $C_0$-semigroup of
angle $\pi/2$, whenever $-A$ does. Moreover, if $-A$  generates a
sectorially bounded holomorphic $C_0$-semigroup of angle greater than $\pi/4$
then $-\psi(A)$ is the generator of a sectorially bounded
holomorphic $C_0$-semigroup as well, \cite[Proposition 7.4]{Laub}.
However, in the latter case, the relations between the sectors  of
holomorphy of the two semigroups was not made precise in
\cite{Laub}.

An affirmative answer to the Kishimoto-Robinson question for
uniformly convex Banach spaces $X$ was obtained
 in \cite{Mir1}-\cite{Mir3} using Kato-Pazy's characterization of holomorphic $C_0$-semigroups on uniformly convex spaces.
 In fact, a positive answer to the question  in its full generality was also claimed in \cite{Mir2}.
 However, there seem to be an error in the
arguments there (see Remark \ref{gapMir} for more on that), and
moreover the permanence of sectors  and thus the sectorial
boundedness of semigroups has not been addressed  in
\cite{Mir1}-\cite{Mir3}.

Another  class  of  problems related to $\mathcal M$ concerns
 Bernstein functions $\psi$ yielding semigroups $(e^{-t\psi(A)})_{t \ge 0}$ with  better properties
 than the initial semigroup $(e^{-tA})_{t \ge 0},$
  as in Yosida's example with $\psi(z)=z^\alpha,\alpha \in (0,1).$ In particular, it is of value
to know when Bernstein functions transform generators of bounded
$C_0$-semigroups into generators of bounded holomorphic
$C_0$-semigroups. (Here the boundedness of semigroup is assumed
only on the real half-line.) This property of Bernstein functions
will further be referred to as the \emph{improving} property. The
first results on the improving property are due to Carasso and
Kato, \cite{Carasso}. In particular,
 \cite[Theorem 4]{Carasso}
gives a criterion for the improving property of $\psi$ in terms of
the semigroup $(\mu_t)_{t \ge 0}$ corresponding to $\psi$ and also
a necessary condition for that property in terms of $\psi$ itself.
Note that while the characterization for the improving property in
terms of $(\mu_t)_{t \ge 0}$ exists, it can hardly be applied
directly since it is, in general, highly nontrivial to construct
$(\mu_t)_{t \ge 0}$ corresponding to $\psi.$  Thus it is desirable
to have direct characterizations or  conditions for the improving
property of $\psi.$

Certain sufficient conditions for the improving property of $\psi$
were obtained in   \cite{F1}, \cite{Saddi}, \cite{Mir2} and
\cite{Mir4}. Nontrivial  applications of the improving conditions
from \cite{Carasso} and \cite{F1} can be found in \cite{F2},
\cite{C}
and  \cite{Kov}. We note also \cite{Du} where
similar results were obtained in a discrete setting.

Our approach to the two problems on $\mathcal M$ mentioned above
relies on certain extensions of the theory of Bernstein functions
and its applications to operator norm estimates  by means of
functional calculi. Observe that the problems are comparatively
simple if $\psi$ is a \emph{complete} Bernstein function,
\cite{Laub}. Thus it is natural to try to use this partial answer
in a more general setting of  Bernstein functions. Our main idea
relies on comparing a fixed Bernstein function $\psi$ to a
complete Bernstein function $\varphi$ associated to $\psi$ in a
unique way. It appears that the functions $\psi$ and $\varphi$ are
intimately related and the behavior of $\psi$ and its transforms
match in a natural sense the behavior of $\varphi$ and the
corresponding transforms. So our aim is to show that for
appropriate $\lambda$ the ``resolvent'' functions
$(\psi+\lambda)^{-1}$ and $(\varphi+\lambda)^{-1}$ differ by a
summand with good integrability (and other analytic) properties
and then to recast this fact in terms of functional calculi. The
latter step is not however direct and to perform it correctly and
transparently we have to use an interplay between several
well-known calculi. Apart from answering the questions from
\cite{Rob} and \cite{Laub}, another advantage of our approach is
that we have a good control over fine properties of $\psi(A),$
thus deriving the property of permanence of angles under the map
$\mathcal M.$

Our  functional calculus approach leads, in particular,  to the
following statement which is one of the main results in this paper.

\begin{thm}\label{TangleInt}
Let $-A$ be the generator of a bounded holomorphic $C_0$-semigroup
of angle  $\theta \in (0,\pi/2]$ on a Banach space $X.$ Then for
every Bernstein function $\psi$ the  operator $-\psi(A)$ generates
a bounded holomorphic $C_0$-semigroup of angle $\theta$ on $X$ as
well. Moreover, if $-A$ generates a sectorially bounded holomorphic
$C_0$-semigroup of angle $\theta,$ then the same is true for
$-\psi(A).$
\end{thm}

The functional calculus ideas allow one also to \emph{characterize}  the
improving property of $\psi$ if $\psi$ is a complete Bernstein
function, i.e. if, in addition, $\psi$ extends to the upper
half-plane and maps it into itself. The characterization given in Corollary \ref{IlCBF12} below is a consequence of the following interesting result (see Theorem \ref{IlCBF11}).
\begin{thm}\label{CBFIntro}
Let $\psi$ be  a complete Bernstein function and let $\gamma \in
(0,\pi/2)$ be fixed. The the following assertions are equivalent.
\begin{itemize}

\item [(i)] the function
$ \psi$ maps the closed  right half-plane  into the sector
 $\overline{\Sigma}_\gamma:=\{\lambda\in \mathbb C: |\arg (\lambda)|\le \gamma\}.$

\item [(ii)]  For each (complex) Banach space $X$ and each generator $-A$ of a bounded
$C_0$-semigroup on $X,$ the operator $-\psi(A)$ generates  a
sectorially bounded holomorphic $C_0$-semigroup on $X$ of angle
$\pi/2-\gamma.$

\end{itemize}
\end{thm}

Moreover, we are able to strengthen essentially the results by
Fujita from \cite{F1} removing in particular several assumptions
made in \cite{F1}.
\begin{thm}\label{BFIntro1}
Let $\psi$ be  Bernstein function. Suppose there exist $\theta\in
(\pi/2,\pi)$ and $r>0$ such that $\psi$ admits a continuous extension 
to $\overline{\Sigma}_\theta$ which is holomorphic in $\Sigma_\theta,$ and
\begin{equation}\label{CKrIntro}
0\le \arg (\psi(\lambda)) \le \pi/2 \qquad \mbox{if}\qquad 0 \le
\arg (\lambda) \le \theta \quad \text{and} \quad |\lambda|\ge r.
\end{equation}
If $-A$ is the generator of a  bounded $C_0$-semigroup on $X$,
then the (bounded) $C_0$-semigroup $(e^{-t\psi(A)})_{t\ge 0}$ is
holomorphic in $\Sigma_{\theta_0}$ with
$\theta_0=\frac{\pi}{2}(1-\pi/(2\theta)).$
\end{thm}

Finally, let us describe the structure of our paper. The paper is
organized is as follows. Section  \ref{Secberfunct} contains basic
information on Bernstein functions together with new notions and
properties which are related to Bernstein functions and
 are  crucial for the sequel. In Section  \ref{Fun}, we
review functional calculi theory needed for the proofs of our main
results. Several estimates for resolvents  of operators given by
complete Bernstein functions of semigroup generators are contained
in Section \ref{frac}. They  are probably of some independent
interest. Section \ref{Main1} is devoted to the proof of one of
our central results, Theorem \ref{TangleInt}. In Section
\ref{Main2}, we study the improving properties of Bernstein
functions and complement and strengthen the corresponding
statements by Carasso-Kato and Fujita.  Finally, in Appendix, we comment on alternative ways
to prove Theorem \ref{TangleInt}.

\section{Notations and generalities}

For a closed linear operator $A$ on a complex Banach space $X$ we
denote by $\dom(A),$ $\ran(A),$ $\rho(A)$ and $\sigma(A)$ the {\em
domain}, the {\em range}, the {\em resolvent set} and the {\em
spectrum} of $A$, respectively, and let  $\cls{\ran}(A)$ stand for
the norm-closure of the range of $A$. For closable $A$ we denote
its closure by $\overline{A}.$
 The space of bounded linear operators
on $X$ is denoted by $\mathcal L(X)$.

The Laplace transform  $\widehat \mu$ of a Laplace transformable
measure  $\mu$ will be defined as usual as
 $$ \widehat \mu(z)
:= \int_0^{\infty} e^{-sz} \, \mu(d{s})
$$
for appropriate $z.$ The same notation and definition will be
clearly apply for Laplace transformable functions as well.

 For linear operators $A$ and $B$ on $X$, as usual, we consider  the sum $A+B$ and product  $AB$ with domains given by
\begin{eqnarray*}
\dom(A+B) &=& \dom(A) \cap \dom(B), \\
\dom(AB) &=& \{x \in \dom(B): Bx \in\dom(A)\}.
\end{eqnarray*}

We will write the Lebesgue integral $\int_{(0,\infty)}$ as
$\int_{0+}^{\infty}$.
The symbol $*$ will denote convolution of measures
(or functions).

Let also
\[
\mathbb C_{+}=\{\lambda \in \C:\,{\rm Re}\,\lambda>0\},\qquad
\mathbb R_+=[0,\infty),
\]
and for $\beta\in (0,\pi]$ and $R>0$, we denote
\[
\Sigma_{\beta}:=\{z\in \C:\,|\arg \lambda|<\beta\},\qquad
\Sigma_0=(0,\infty),
\]
and
\[
\Sigma_{{\beta}}(R):=\{z\in \Sigma_\beta:\, |z|<R\}, \qquad
\Sigma_{\beta}^{+}:=\{\lambda\in \C:\, 0<\arg \lambda<\beta\}.
\]
Finally, we let
\[
H^{+}:=\{\lambda\in \C: {\rm Im}\,\lambda>0 \}.
\]

\section{Bernstein functions}\label{Secberfunct}

This section will lay a function-theoretical background for our
functional calculi considerations in the subsequent sections. In
particular, we will prove a number of new properties of Bernstein
functions and revisit some of known ones crucial for the sequel.

We start with recalling one of possible definitions of a Bernstein
function.
\begin{defn}
A smooth function $\psi: [0,\infty)\mapsto [0,\infty)$ is called
\emph{Bernstein} if its derivative $\psi'$ is \emph{completely
monotone}, i.e.
\begin{equation}
\psi'(\lambda)=\int_{0}^{\infty}e^{-\lambda s}\, \nu (ds), \qquad
\lambda >0,
\end{equation}
for a  Laplace transformable positive Radon measure $\nu$ on
$[0,\infty).$
\end{defn}
The class of Bernstein functions will be denoted by
$\mathcal{BF}.$

By \cite[Theorem 3.2]{SchilSonVon2010}, $\psi$ is Bernstein if and
only if there exist  $a, b\geq 0$ and a positive Radon measure
$\mu$ on $(0,\infty)$ satisfying
\begin{equation*}
\int_{0+}^\infty\frac{s}{1+s}\,\mu(d{s})<\infty \label{mu}
\end{equation*}
such that
\begin{equation}\label{defBF}
\psi(z)=a+bz+\int_{0+}^\infty (1-e^{-zs})\mu(d{s}), \qquad z>0.
\end{equation}
The formula \eqref{defBF} is called the L\'evy-Hintchine
representation of $\psi.$ The triple $(a,b,\mu)$ is defined
uniquely and is called L\'evy triple of $\psi.$ We will then often
write $\psi \sim (a,b,\mu)$ meaning the L\'evy-Hintchine
representation of $\psi.$ Every Bernstein function extends
analytically to $\mathbb C_+$ and continuously to
$\overline{\mathbb C}_+$
In the following  Bernstein functions will be identified with
their continuous extensions to $\overline{\mathbb C}_+.$

The standard examples of Bernstein functions include
\begin{equation}\label{exber}
1-e^{-z}, \,\, \log (1+z), \,\, \text{and}\,\, z^{\alpha}, \,
\alpha \in [0,1].
\end{equation}

There is a profound theory of Bernstein functions with many
implications in functional analysis and probability theory. Being
unable to give any reasonable account of them, we refer to a recent book
\cite{SchilSonVon2010}.

Geometric  properties of Bernstein functions will be  of
particular importance for us, in particular, the fact
 that a Bernstein function $\psi$ preserve sectors $\Sigma_\omega$ in
 a sense that $\psi(\overline{\Sigma}_{\omega})\subset \overline{\Sigma}_\omega,$
 see   \cite[Proposition 3.6]{SchilSonVon2010} and \cite[Corollary 3.3]{Sonia}.
For a later use, we state this
result as a proposition below and provide it with a simple proof using an idea from
\cite{Sonia}.

\begin{prop}\label{PrB}
Let  $F\in \mathcal{BF}$. Then
 $F$ preserves angular sectors, i.e.
\begin{equation}\label{prsec}
F(\overline{\Sigma}_\omega)\subset
\overline{\Sigma}_\omega,\qquad \omega\in (0,\pi/2).
\end{equation}
\end{prop}
\begin{proof}
Note that
\[
1-e^{-z}=\frac{2 z}{\pi} \int_0^\infty \frac{(1-\cos u)\,du}{u^2+z^2},\qquad
{\rm Re}\,z>0.
\]
Since for every  $u>0$ the function
\[
g_u(z):=\frac{z}{u^2+z^2}
=\left(z+\frac{u^2}{z}\right)^{-1},\qquad {\rm Re}\,z>0,
\]
preserves the angular sectors, we infer that the function
\[
z\mapsto 1-e^{-z},\quad z\in \overline{\C}_{+}
\]
preserves the angular sectors too.
Then, using the Levy-Khinchine representation \eqref{defBF} of $\psi$, we obtain (\ref{prsec}).
\end{proof}

The preservation of sectors property has further consequences
important for the proofs of our main assertions. One of them is
mentioned below together with another ``geometric'' property.
\begin{prop}\label{P1}
Let $\psi\in \mathcal{BF}.$
\begin{itemize}
\item [(i)] For all $\gamma>0$, $\beta\in (0,\pi/2)$ such
that  $\gamma+\beta<\pi$,
\begin{equation}\label{FEH}
|z+\psi(\lambda)|\ge \cos ((\gamma+\beta)/2)\,
(|z|+|\psi(\lambda)|),\quad z\in \overline{\Sigma}_\gamma,\quad  \lambda\in \overline{\Sigma}_\beta.
\end{equation}
\item [(ii)] one has
\begin{equation}\label{Re}
{\rm Re}\,\psi(\lambda)\ge \psi({\rm Re}\lambda),\quad \lambda\in \C_{+}.
\end{equation}
\end{itemize}
\end{prop}
\begin{proof}
Note first that if $\beta\in (-\pi,\pi)$ and $s>0$ then
\begin{equation}\label{number}
|1+se^{i\beta}|^2=1+s^2+2s\cos\beta \ge\cos^2\beta/2 (1+s)^2.
\end{equation}
Let now $ \gamma>0,\quad \beta>0,\quad \gamma+\beta<\pi, $ and
\[
z=r e^{i\gamma_0}\in \overline{\Sigma}_\gamma,\quad \lambda= \rho
e^{i\beta_0}\in \overline {\Sigma}_\beta, \quad |\gamma_0|\le
\gamma,\quad |\beta_0|\le \beta.
\]
Then, using (\ref{number}), we obtain
\[
|z+\lambda|=r|1+r^{-1}\rho e^{i(\beta_0-\gamma_0)}|\ge
\cos((\beta_0-\gamma_0)/2)\,(|z|+|\lambda|).
\]
From this, since
\[
|\beta_0-\gamma_0|\le \beta+\gamma \in (0,\pi), \quad \text{ and}
\quad \cos((\beta_0-\gamma_0)/2)\ge \cos((\beta+\gamma)/2),
\]
it follows that
\begin{equation}\label{number1}
|z+\lambda|\ge \cos((\beta+\gamma)/2)\,(|z|+|\lambda|),\quad z\in
\overline{\Sigma}_\gamma,\quad \lambda\in \overline{\Sigma}_\beta.
\end{equation}
Now  (i) is a direct consequence of Proposition \ref{PrB} and (\ref{number1}). To prove (ii) it suffices to note that
\[
{\rm Re}\,(1-e^{-\lambda})\ge 1-e^{-{\rm Re}\,\lambda},\quad
\lambda\in \C_{+},
\]
and use the L\'evy-Hintchine representation for $\psi.$
\end{proof}

Recall that every Bernstein function $\psi$ satisfies
\begin{equation}\label{Jac1}
\psi(\lambda)\le \psi(s\lambda)\le s\psi(\lambda),\qquad
\lambda>0,\quad s\ge 1,
\end{equation}
and
\begin{equation}\label{Jac2}
\lambda^k \psi^{(k)}(\lambda)\le k!\psi(\lambda),\qquad
\lambda>0,\quad k\in \N,
\end{equation}
see \cite[p. 205]{Jacob}.

The next lemma provides growth estimates for holomorphic extensions of Bernstein functions to the right half-plane.
\begin{lemma}\label{psi}
Let $\psi\in \mathcal{BF}$.
\begin{itemize}
\item [(i)]
For all  $\beta\in(0,\pi/2)$ and
$t>0,$
\begin{equation}\label{R1}
|\psi(te^{\pm i\beta})|\ge \psi(t\cos\beta).
\end{equation}
\item [(ii)] There exist $c_\psi >0$ such that
\begin{eqnarray}
|\psi(z)|&\le& c_\psi |z|,\quad z\in \overline{\C}_{+},\quad |z|\ge 1.
\end{eqnarray}
\item [(iii)] For all  $\beta\in(0,\pi/2)$
\begin{eqnarray}\label{3.13}
 |\psi(z)|&\ge& |z|\psi'(1)\cos\beta, \quad z\in \overline{\Sigma}_\beta,\quad |z|\le 1.
\end{eqnarray}
\end{itemize}
\end{lemma}

\begin{proof}
Let $\psi \sim (a, b,\mu)$. Then
\begin{eqnarray*}
|\psi(te^{\pm i\beta})|&\ge& {\rm Re}\,\psi(te^{i\beta})\\
&=& a+bt\cos\beta+
\int_{0+}^\infty {\rm Re}(1-e^{-ste^{i\beta}})\,\mu(ds)
\\
&\ge& a+bt\cos\beta+ \int_{0+}^\infty
(1-e^{-st\cos\beta})\,\mu(ds)= \psi(t\cos\beta),
\end{eqnarray*}
that is (\ref{R1}) holds.

To prove (ii), we
note that (\ref{defBF}) yields
\[
|\psi(z)|\le a+b|z|+|z|\int_{0+}^1
s\,\mu(\ud{s})+2\int_1^\infty \,\mu(\ud{s}),\qquad z\in
\overline{\C}_{+},
\]
and then
\begin{equation}\label{PP2}
|\psi(z)|\le c_\psi |z|,\qquad z\in \overline{\C}_{+},
\quad |z|\ge 1.
\end{equation}

Furthermore, by (i) and (\ref{Jac1}), we have for any $\beta\in (0,\pi/2):$
\begin{equation*}\label{poch0}
|\psi(z)|\ge \cos\beta \psi(|z|), \qquad z\in \overline{\Sigma}_\beta,\quad
|z|\le 1.
\end{equation*}
Since by (\ref{Jac2}),
\begin{equation*}\label{poch}
\psi(z)\ge z\psi'(z)\ge \psi'(1)z,\qquad z\in (0,1],
\end{equation*}
 we obtain \eqref{3.13}.
\end{proof}

One more notion related to Bernstein functions will also be needed
in the sequel. Let us recall (see e.g. \cite[Definition 5.24]
{SchilSonVon2010}) that a function $f: (0,\infty)\mapsto
(0,\infty)$ is said to be potential, if $f=1/\psi$, where $\psi\in
\mathcal{BF}$. The set of all potentials will be denoted by
$\mathcal{P}.$ Note that $\mathcal{P}$ consists precisely of
 completely monotone functions $f$ satisfying $1/f\in \mathcal{BF}$.

It is often convenient to restrict one's attention to  a rich
subclass of Bernstein functions formed by complete Bernstein
functions. It  has a rich structure which makes it especially
useful in applications.  A Bernstein function $\psi$ is said to be
{\it a complete Bernstein function} if the measure $\mu$ in its
L\'evy-Hintchine representation (\ref{defBF}) has a completely
monotone density $m$ with respect to Lebesgue measure.
The set of all complete Bernstein functions will be denoted by
$\mathcal{CBF}$ .

The class of complete Bernstein functions allows a number of
characterizations. The ones relevant for our purposes are
summarized in the following statement, see e.g. \cite[Theorem 6.2]{SchilSonVon2010}).

\begin{thm}\label{Shill}
Let  $\psi$ be a non-negative function on $(0,\infty)$. Then the
following conditions are equivalent.
\begin{itemize}

\item [(i)] \,$\psi\in \mathcal{CBF}$,

\item [(ii)] \, There exists a Bernstein function $\varphi$ such that
\begin{equation}\label{Ac1}
\psi(\lambda)=\lambda^2 \widehat \varphi(\lambda),\quad \lambda>0.
\end{equation}

\item [(iii)]
$\psi$ admits a holomorphic extension to $H^+$ such that
\[
{\rm Im}\,(\psi(\lambda))\ge 0\quad \mbox {for all}\quad \lambda \in
H^+,
\]
and such that the limit
\[
\psi(0+)=\lim_{\lambda\to 0+}\,\psi(\lambda)
\]
exists.

\item [(iv)]\,$\psi$ admits a holomorphic extension  to $\C\setminus (-\infty,0]$
which is given by
\begin{equation}\label{Cbf}
\psi(\lambda)=a+b\lambda+\int_{0+}^\infty
\frac{\lambda\,\sigma(ds)}{\lambda+s},
\end{equation}
where $a,b\ge 0$  and $\sigma$ is a positive Radon measure on
$(0,\infty)$ such that
\begin{equation}\label{mmu}
\int_{0+}^\infty \frac{\sigma(ds)}{1+s}<\infty.
\end{equation}
The triple $(a,b,\sigma)$ is defined uniquely and it is called the
Stieltjes representation of $\psi.$
\end{itemize}
\end{thm}
Using the above result it is easy to see that the first function
in \eqref{exber} is not complete Bernstein, while the other
Bernstein functions there are clearly complete.

The next statement sharpens Proposition \ref{PrB} in a specific
situation when complete Bernstein function has its range in a sector
smaller than the right half-plane.
\begin{prop}\label{psi1}
Let $\psi\in \mathcal{CBF}$ and suppose that
\begin{equation}\label{Pi}
\psi(\overline{\C}_+)\subset \overline{\Sigma}_\gamma
\end{equation}
for some $\gamma\in (0,\pi/2).$ Let $\theta_0 \in (\pi/2, \pi)$ be
defined by
\begin{equation}\label{CCcot}
|\cos\theta_0|=\frac{\cot \gamma}{1+\cot \gamma}.
\end{equation}
Then for every $\theta\in (\pi/2,\theta_0)$ one has
\begin{equation}\label{Pi1}
\psi(\overline{\Sigma}_\theta)\subset
\overline{\Sigma}_{\tilde{\theta}},
\end{equation}
where
\begin{equation*} \cot
\tilde{\theta}=\frac{1+\cot\gamma}{\sin\theta}\left(\frac{\cot
\gamma}{1+\cot \gamma}-|\cos\theta|\right), \qquad
\tilde{\theta}\in  (0,\pi/2).
\end{equation*}
\end{prop}

\begin{proof}
By (\ref{Pi}) it follows that $\psi$ has the Stieltjes representation $(a,0,\sigma).$
Note that
\begin{eqnarray}\label{psiP}
\psi(re^{i\theta})&=&a+\int_0^\infty \frac{r(r+
t\cos\theta)\,\sigma(dt)}
{r^2+t^2+2rt\cos\theta} \\
&+&i\sin \theta\int_0^\infty \frac{r t\,\sigma(dt)}
{r^2+t^2+2rt\cos\theta},\quad r>0,\quad |\theta|<\pi, \notag
\end{eqnarray}
and
\[
{\rm Im}\,\psi(re^{i\theta})>0,\quad r>0,\quad \theta\in (0,\pi).
\]
Setting in \eqref{psiP}
the value $\theta=\pi/2$ and using \eqref{Pi}, we infer that
\begin{equation}\label{imply}
a+\int_0^\infty \frac{r^2\,\sigma(dt)} {r^2+t^2}\,\ge \cot \gamma \,
\int_0^\infty \frac{rt\,\sigma(dt)} {r^2+t^2}, \qquad r>0.
\end{equation}

Moreover, note that  for every $\theta\in [\pi/2,\pi),$ and  all
$r,t>0,$
\begin{equation}\label{inref1}
\frac{1}{r^2+t^2}\le \frac{1}{r^2+t^2+2rt\cos\theta} \le
\frac{1}{(1-|\cos\theta|)(r^2+t^2)}, \qquad r,t>0.
\end{equation}
Hence, if $\theta \in [\pi/2,\theta_0],$ where $\theta_0$ is given
by \eqref{CCcot}, then by  \eqref{psiP},
\eqref{inref1} and  \eqref{imply} we obtain
\begin{eqnarray*}
{\rm Re}\,\psi(re^{i\theta})&\ge& a+\int_0^\infty \frac{r^2-r
t|\cos\theta|} {r^2+t^2+2rt\cos\theta}\,\sigma(dt)
\\
&\ge& a+\int_0^\infty \frac{r^2\,\sigma(dt)}{r^2+t^2}-
\frac{|\cos\theta|}{1-|\cos\theta|} \int_0^\infty \frac{r
t\,\sigma(dt)}{r^2+t^2}
\\
&\ge& \left(\cot \gamma -\frac{|\cos\theta|}{1-|\cos\theta|}
\right) \int_0^\infty \frac{r t} {r^2+t^2}\,\sigma(dt)
\\
&\ge& \left(\cot \gamma
-\frac{|\cos\theta|}{1-|\cos\theta|}\right)(1-|\cos\theta|)
\int_0^\infty \frac{r t\,\sigma(dt)} {r^2+t^2+2rt\cos\theta}
\\
&=&\alpha(\theta)\,{\rm Im}\,\psi(re^{i\theta}),
\end{eqnarray*}
where
\[
\alpha (\theta)=\frac{(1-|\cos\theta|)}{\sin\theta}
\left(\cot\gamma-\frac{|\cos\theta|}{1-|\cos\theta|}\right)
=\frac{1+\cot\gamma}{\sin\theta}\left(\frac{\cot \gamma}{1+\cot
\gamma}-|\cos\theta|\right).
\]
Note that
\[
\alpha(\theta_0)=0,\quad \alpha(\pi/2)=\cot\gamma\quad \mbox{and}
\quad \alpha(\theta)>0 \quad \mbox{if}\quad \theta\in
[\pi/2,\theta_0).
\]
Moreover,
\[
\alpha'(\theta)=\frac{\cot\gamma|\cos\theta|-(1+\cot\gamma)}{\sin^2\theta}
\le\frac{1}{\sin^2\theta}<0, \quad \theta\in [\pi/2,\theta_0],
\]
hence $\alpha(\theta)$ is positive  and decreasing on
$[\pi/2,\theta_0]$. Therefore, for all  $\theta'\in
(\pi/2,\theta)$ and $\theta\in (\pi/2,\theta_0)$ we have
\[
{\rm Re}\,\psi(re^{i\theta'})\ge \alpha(\theta') {\rm Im}\,\psi(re^{i\theta'}) \ge
\alpha(\theta) {\rm Im}\,\psi(re^{i\theta'}).
\]
On the other hand, if $\theta'\in (0,\pi/2]$ then, by our
assumption,
\[
{\rm Re}\,\psi(re^{i\theta})\ge \cot \gamma \,{\rm Im}\,\psi(re^{i\theta'})
\ge \alpha(\theta) {\rm Im}\,\psi(re^{i\theta'}).
\]
Thus,
\[
\psi(\overline{\Sigma}_\theta^{(+)})\subset
\overline{\Sigma}_{\tilde{\theta}}^{(+)},
\]
and, in view of  $\psi(re^{-i\theta})=\overline{\psi(re^{i\theta})}$,  the assertion \eqref{Pi1} follows.
\end{proof}

A direct consequence of Theorem \ref{Shill}, (iii), is that
\begin{equation*}
\psi \in \mathcal{CBF}, \psi \not\equiv 0, \quad \Leftrightarrow
\quad \lambda/\psi(\lambda) \in \mathcal{CBF}\quad \Leftrightarrow
\quad  \lambda \psi(1/\lambda) \in \mathcal{CBF}.
\end{equation*}
Thus,  $0\not=h \in \mathcal{CBF}$ if and only if
$\varphi(\lambda)=\lambda h(1/\lambda) \in \mathcal{CBF}.$
In view of  the latter property and Theorem \ref{Shill}, (ii)  the
following
 definition is natural.

\begin{defn}\label{Im}
A function $\varphi\in \mathcal{CBF}$ is said to be associated
with $\psi\in \mathcal{BF}$ if
\begin{equation}\label{Ac3}
\varphi(\lambda)=\lambda^{-1}\widehat \psi (\lambda^{-1}),\quad
\lambda>0.
\end{equation}
\end{defn}
The notion of  associated complete Bernstein function will be of
primary importance in this paper, and we will first collect its
several properties in Lemma \ref{Ass1} below. To this aim, the
next auxiliary lemma will be useful.

\begin{lemma}\label{Rr}
Define
\begin{equation}\label{dd}
\Delta(\lambda):=\frac{1}{1+\lambda}-e^{-\lambda},\qquad
\lambda\in\C_{+}.
\end{equation}
Then
\begin{equation}\label{Rest}
|\Delta(\lambda)|\le \frac{4\,|\lambda|^2}{(1+{\rm
Re}\,\lambda)^3},\qquad \lambda\in \C_{+}.
\end{equation}
\end{lemma}

\begin{proof}
We use the integral representation from \cite[p. 3056, Eq. 4.21)]{GT}:
\begin{equation}\label{Delta}
\Delta(\lambda)=\lambda^2\int_0^\infty e^{-\lambda s} G(s)\,ds,
\qquad \lambda \in \C_+,
\end{equation}
where
\[
G(s)=\chi(1-s)(s-1+e^{-s})+\chi(s-1)e^{-s},\qquad s>0,
\]
and $\chi$ stands for the characteristic function of $(0,\infty).$
Since
\[
s-1+e^{-s}\le \frac{s^2}{2},\qquad e^{-s}\le
\frac{2}{(s+1)^2},\qquad s>0,
\]
 (\ref{Delta}) implies that
\begin{eqnarray*}
|\lambda|^{-2} |\Delta(\lambda)|&\le& \int_0^1 e^{-s {\rm Re}\,
\lambda}(s-1+e^{-s})\,ds+ \int_1^\infty e^{-s {\rm Re}\, \lambda}
e^{-s}\,ds
\\
&\le& \frac{e}{2} \int_0^1 e^{-s({\rm Re}\, \lambda+1)} s^2\,ds+
\frac{e^{-{\rm Re}\, \lambda-1}}{{\rm Re}\, \lambda+1}
\\
&\le &\frac{4}{({\rm Re}\, \lambda +1)^3}, \qquad  \lambda\in
\C_{+}.
\end{eqnarray*}
\end{proof}
\begin{lemma}\label{Ass1}
Let $\varphi\in \mathcal{CBF}$ be  associated with $\psi\in
\mathcal{BF}$ and let
\begin{equation}\label{FF}
\psi(\lambda)=a+b\lambda+\int_{0+}^\infty (1-e^{-\lambda s})\,\nu(ds),\quad \lambda>0.
\end{equation}
Then
\begin{itemize}
\item [a)]\,  $\varphi$ has the representation
\begin{equation}\label{ARep}
\varphi(\lambda)=a+b\lambda+\int_{0+}^\infty \frac{\lambda
s\,\nu(ds)}{1+\lambda s},\quad \lambda>0.
\end{equation}
\item [b)] \, the inequality
\begin{equation}\label{ARep1}
{\rm Re}\,\psi(\lambda)\ge \varphi({\rm Re}\,\lambda),\quad
\lambda\in \C_{+},
\end{equation}
holds.
\item [c)]\, the estimate
\begin{equation}\label{L12}
|\psi(\lambda)-\varphi(\lambda)|\le 2 |\lambda|^2 \varphi''({\rm
Re}\,\lambda),\quad \lambda\in \C_{+},
\end{equation}
\end{itemize}
holds.

$d)$\,  $\psi$ is bounded if and only if  $\varphi$ is  bounded,
and then for any $\beta\in (0,\pi/2),$
\[
\lim_{\lambda\to\infty,\,\lambda\in \Sigma_\beta}\,\psi(\lambda)=
\lim_{\lambda\to\infty,\,\lambda\in
\Sigma_\beta}\,\varphi(\lambda).
\]
\end{lemma}

\begin{proof}
The assertion $a)$ follows directly from (\ref{FF}) and
(\ref{Ac3}).

To prove $b)$ we note that
\[
1-e^{-\tau}\ge \frac{\tau}{1+\tau},\quad \tau>0.
\]
Then, setting $u={\rm Re}\,\lambda>0,$ by \eqref{ARep1},
(\ref{FF}) and (\ref{ARep}), we obtain
\begin{eqnarray*}
{\rm Re}\,\psi(\lambda)\ge \psi(u)
&=& a+b u+\int_{0+}^\infty (1-e^{-u s})\,\nu(ds)\\
&\ge& a+b u+\int_{0+}^\infty \frac{u s\,\nu(ds)}{1+u s}\\
&=&\varphi(u),
\end{eqnarray*}
so that (\ref{ARep1}) holds.

Let us now prove   $c).$ Observe that
\begin{equation}\label{Delta1}
\psi(\lambda)-\varphi(\lambda)=\int_{0+}^\infty \Delta(\lambda s)\,\nu(ds),\quad \lambda\in \C_{+},
\end{equation}
where $\Delta$ is defined by \eqref{dd}, and by \eqref{ARep},
\[
\varphi''(\lambda)=-2\int_{0+}^\infty \frac{
s^2\,\nu(ds)}{(1+\lambda s)^3}.
\]
Then, using (\ref{Rest}), it follows that
\[
|\psi(\lambda)-\varphi(\lambda)|\le \int_{0+}^\infty
|\Delta(\lambda s)|\,\nu(ds) \le 4|\lambda|^2\int_{0+}^\infty
\frac{s^2\,\nu(ds)}{(1+us)^3}\le  2 |\lambda|^2 \varphi''(u),
\]
which is \eqref{L12}.

To prove the first statement in d), it suffices to note that
boundedness of either $\psi$ or $\varphi$ is equivalent to
boundedness of  a measure $\nu$ in \eqref{FF} and \eqref{ARep} by Fatou's theorem. Finally, since
$|\Delta(\lambda)|\le 2$, $\lambda\in \C_{+},$  and
\[
\lim_{\lambda\to\infty,\,\lambda\in
\Sigma_\beta}\,|\Delta(\lambda)|=0,
\]
for any $\beta\in (0,\pi/2),$ \eqref{Delta1} implies the second assertion in  $d)$ by the
bounded convergence theorem.
\end{proof}

One can also give a counterpart of \eqref{L12} with $\varphi''$
replaced by $\varphi'$ as the following corollary of Lemma
\ref{Ass1}, c) shows.

\begin{cor}\label{ACor1}
Let $\varphi\in \mathcal{CBF}$ is  associated with $\psi\in
\mathcal{BF}$. Then for every $\beta\in (0,\pi/2),$
\begin{equation}\label{L22}
|\psi(\lambda)-\varphi(\lambda)|\le \frac{4|\lambda|}{\cos\beta}\,
\varphi'({\rm Re}\,\lambda),\quad \lambda\in \overline{\Sigma}_{\beta}\setminus\{0\}.
\end{equation}
\end{cor}
\begin{proof}
Note that  $\varphi\in \mathcal{CBF}$ implies $s\varphi''(s)\le 2
\varphi'(s)$, $s>0$. Using  (\ref{L12}) and observing that
\[
|\lambda|\le \frac{{\rm Re}\,\lambda}{\cos\beta},\quad \lambda\in
\Sigma_\beta,\quad \beta\in (0,\pi/2),
\]
we arrive at  (\ref{L22}).
\end{proof}

Now we are ready to prove the main result of this section
providing an estimate for the ``resolvents'' of a Bernstein
function and the complete Bernstein function associated to it.

\begin{thm}\label{BTR}
Let $\psi$ be a Bernstein function and let $\varphi$ be the complete
Bernstein function associated with $\psi$. Let  $\omega\in
(\pi/2,\pi)$ and $z\in \Sigma_\omega$ be fixed.
If
\begin{equation}\label{defr}
r(\lambda;z):=\frac{1}{z+\psi(\lambda)}-\frac{1}{z+\varphi(\lambda)},\qquad
\lambda\in \Sigma_{\pi-\omega},
\end{equation}
 then the function $r(\cdot;z)$ is holomorphic in $\Sigma_{\pi-\omega}$
 and for every $\beta \in (0,\pi-\omega):$
\begin{equation}\label{need}
\int_{\partial\Sigma_\beta}|r(\lambda;z)|\frac{|d\lambda|}{|\lambda|}\le
\frac{8}{\cos^2\beta\cos^2((\omega+\beta)/2)\,|z|},\qquad \quad
z\in \Sigma_\omega.
\end{equation}
\end{thm}

\begin{proof}
Note first that $\pi-\omega\in (0,\pi/2)$.
Since by Proposition \ref{PrB}, the functions $\psi$ and $\varphi$ preserve sectors,
$z+\psi$ and $z+\varphi$ are not zero at each point from
$\Sigma_{\pi-\omega}.$   As   $\psi$ and $\varphi$ are holomorphic
in $\mathbb C_+,$ the holomorphicity of $r(\cdot, z)$ in
$\Sigma_{\pi-\omega}$ follows.

Let now
$\beta\in (0,\pi-\omega)$ and
$0\not=\lambda\in \overline{\Sigma}_\beta$, $z\in \Sigma_\omega.$
If
\[
K=\cos((\omega+\beta)/2),
\]
then by Proposition \ref{P1}, (i), we have
\begin{equation}\label{eqqq}
K^2|r(\lambda;z)| \le
\frac{|\varphi(\lambda)-\psi(\lambda)|}
{(|z|+|\psi(\lambda)|)(|z|+|\varphi(\lambda)|)}.
\end{equation}
Let us estimate the numerator and the denominator in the right
hand side of \eqref{eqqq} separately. By \eqref{L22},
\[
|\varphi(\lambda)-\psi(\lambda)|\le \frac{4 |\lambda|}{\cos\beta}\varphi'({\rm Re}\,\lambda),
\]
and, moreover,  (\ref{Re})  and \eqref{ARep1} yield
\begin{eqnarray*}
(|z|+|\psi(\lambda)|)(|z|+|\varphi(\lambda)|)
&\ge& (|z|+{\rm Re}\, \psi(\lambda))(|z|+{\rm Re}\, \varphi(\lambda))\\
&\ge&(|z|+\varphi({\rm Re}\, \lambda))^2),\qquad r>0.
\end{eqnarray*}
Thus, if $\lambda=te^{\pm i\beta}$, $t>0$, then
\begin{equation}\label{Cau}
K^2|r(\lambda;z)| \le \frac{4 t}{\cos\beta}
\frac{\varphi'(t\cos\beta)}{ (|z|+\varphi(t\cos\beta))^2}.
\end{equation}
Hence,
\begin{eqnarray*}
K^2\int_{\partial \Sigma_\beta}
|r(\lambda;z)|\,\frac{|d\lambda|}{|\lambda|} &\le&
\frac{8}{\cos\beta}\int_0^\infty
\frac{\varphi'(t\cos\beta)\,dt}{(|z|+\varphi(t\cos\beta))^2}\\
&\le& \frac{8}{\cos^2\beta\,|z|},
\end{eqnarray*}
and the estimate (\ref{need}) follows.
\end{proof}

\begin{cor}
If $r(\lambda;z)$ is defined as in Theorem \ref{BTR},  then
for all $z\in \Sigma_\omega$ and $\lambda \in \Sigma_{\beta}$:
\begin{equation}\label{funC}
r(\lambda;z)=\frac{1}{2\pi i} \int_{\partial\Sigma_\beta}\,
\frac{r(\mu;z)\, d\mu}{\mu-\lambda},
\end{equation}
and
\begin{equation}\label{SS2}
\int_{\partial \Sigma_\beta}\,\frac{\lambda^k
r(\lambda;z)}{(\lambda+1)^2}\,d\lambda = 0,\qquad k=0,1,
\end{equation}
where the contour $\partial \Sigma_\beta$ is oriented counterclockwise.
\end{cor}
\begin{proof}
If  $\psi$ is unbounded then $\varphi$ is unbounded as well by
Lemma \ref{Ass1}, d), so using (\ref{Cau}), \eqref{Re} and
\eqref{Jac2} we obtain that
\begin{equation}\label{L10}
|r(\lambda;z)|=\mbox{o}(1)\quad \mbox{uniformly in}\quad
\lambda\in \overline{\Sigma}_\beta,\quad \lambda\to\infty,
\end{equation}
for any $z\in \Sigma_\omega$. If  $\psi$ is bounded then
(\ref{L10}) follows directly from (\ref{eqqq}) and Lemma
\ref{Ass1}, d).

Now (\ref{L10}), (\ref{need}) and a standard argument based on
Cauchy's integral formula yield the representation \eqref{funC}.

Finally, \eqref{SS2} is a consequence of (\ref{L10}) and
(\ref{need}).
\end{proof}

Finally, we mention  a property of Bernstein functions which is at
the heart of the notion of subordination. To formulate it, recall
that a family of positive Radon measures $(\mu_t)_{t\ge 0}$ on
$[0,\infty)$ is called a vaguely continuous convolution semigroup
of subprobability measures if for all $t,s \ge 0,$
\begin{equation}\label{vague}
\mu_t([0,\infty))\le 1, \qquad \mu_{t+s}=\mu_t*\mu_s, \qquad
\text{and} \qquad \mbox{vague}-\lim_{t\to0+}\mu_t=\delta_0,
\end{equation}
where $\delta_0$ stands for the Dirac measure at zero.  The following classical result due to Bochner
can be found e.g. in \cite[Theorem 5.2]{SchilSonVon2010}.

\begin{thm}\label{Bochner}
The function $\psi:[0,\infty)\to (0,\infty)$ is Bernstein if and
only if there exists a vaguely continuous convolution semigroup
of subprobability measures $(\mu_t)_{t \ge 0}$ on $[0,\infty)$  such that
\begin{equation}\label{CMonG}
\widehat \mu_t(z)=\int_0^\infty e^{-z s}\,\mu_t(ds)=e^{-t\psi(z)},
\qquad z \ge 0,
\end{equation}
for all $t \ge 0.$
\end{thm}
Note that $(\mu_t)_{t \ge 0}$ above is defined uniquely.

\section{Preliminaries on functional calculi}\label{Fun}

The following discussion of functional calculi may seem to be rather long.
However, our arguments depend on all of those calculi essentially,
and we do not see any other way to make the arguments transparent
 than to introduce the calculi and to explore the relations
between them.

\subsection{Abstract functional calculus and its extensions}\label{abstrfun}
We start from very abstract considerations. However, such an
approach will allow to present several functional calculi below in
a unified manner. For its more comprehensive exposition we refer
to \cite[Section 1]{Haa2006}.

Let $M$ be a commutative algebra with unit $1,$  $N \subset M$  be
its subalgebra, and $X$ be a complex Banach space. Let $\Phi: N
\mapsto \mathcal L(X)$ be a homomorphism.  Then the triple $(M, N,
\Phi)$ is called a (primary) functional calculus over $X.$ If the
set ${\rm Reg}(N) := \{e \in N: \Phi(e) \, \text{is injective}\}$
is not empty, then each member of ${\rm Reg}(N)$ is called a
regulariser. If  $f \in  M$ and there is $e \in {\rm Reg}(N)$ such
that also $ef \in N,$ then $f$ is called regularisable and $e$ a
regulariser for $f.$ If ${\rm Reg}(1)\not=\emptyset$, then the functional calculus is called {\em proper}.
Clearly, if the functional calculus is proper then $\mathcal N:= \{f \in  M : f
\,\text{is regularisable}\}$
 is a subalgebra of $M$ containing $N,$
and for $f \in \mathcal N$  we then define
\begin{align}\label{abstrdef}
{\rm dom}\, (\Phi_e(f)):=&\{x \in X :
(ef)(A)x \in {\rm ran}\,\Phi(e) \}\\
\Phi_e(f) :=& \Phi(e)^{-1} \Phi(ef) \notag
\end{align}
where $e  \in {\rm Reg}(N)$ is a regulariser for $f.$ Then
$\Phi_e(f)$ is a well-defined closed linear operator on $X,$ and the
definition \eqref{abstrdef} is independent of the  regulariser
$e.$ The mapping $\Phi_e: \mathcal N \ni f \mapsto \Phi(f)$
defined in \eqref{abstrdef} extends $\Phi,$ and one usually writes
$f$ instead of $\Phi_e(f).$ The triple $(M, \mathcal N, \Phi_e)$
is called the extended functional calculus (meaning the extension
of $\Phi$ from $N$ to $\mathcal N.$) We will use the same
terminology if $\mathcal N$ is replaced by any of its subalgebras
containing $N.$

The extended functional calculus has a number of natural
(expected) properties of functional calculi, and we stress here
two of them which will be used regularly in the sequel.

\begin{prop}\label{funcal}
 Let $(M, \mathcal N, \Phi_e)$ be an extended  functional calculus over a
Banach space $X.$ Then the following assertions hold.
\begin{itemize}
\item [(i)] If $B \in \mathcal L(X)$ commutes with each $\Phi_e (e), e \in  {\rm Reg}(N),$ the it commutes with each
 $\Phi_e(f), f \in \mathcal N.$
\item [(ii)]Sum rule: given $f, g \in \mathcal N$ one has
$\Phi_e(f) + \Phi_e(g) \subset \Phi_e(f + g),$ with the equality
if $\Phi_e(g) \in \mathcal L (X);$
\item [(iii)] Product rule: given $f,g \in
\mathcal N$ one has $\Phi_e (f)\Phi_e (g)\subset \Phi_e (fg),$
with the equality if $\Phi_e (g) \in \mathcal L (X).$
\end{itemize}
\end{prop}
From Proposition \ref{funcal}, (iii) it follows that if $f$ is
regularizable and $e$ is a regularizer, then
\begin{equation}\label{domian}
{\rm ran}\,(\Phi_e(e))\subset {\rm dom}\,(\Phi_e(f)).
\end{equation}

\subsection{Sectorial operators and holomorphic functional
calculus}\label{funsec}

There are several ways to define a function of a sectorial
operator. Probably the most well-known approach to that task is
provided by the holomorphic functional calculus. This calculus
will be relevant for us, and we will set it up below omitting some
crucial details and referring to \cite[Sections 1-2]{Haa2006} for
more information.

\begin{defn}
A closed linear operator $A$ on $X$ is called sectorial of angle
$\omega\in [0,\pi)$ (in short: $A\in \mbox{Sect}(\omega)$) if
$\sigma(A)\subset \overline{\Sigma}_\omega$ and
\[
M(A,\omega_0):= \sup\{\|\lambda (\lambda-A)^{-1}\|: \lambda\not\in
\Sigma_{\omega_0}\}<\infty
\]
for all $\omega<\omega_0<\pi$.
The number
\[
\omega_A:=\min\,\{\omega:\,A\in\mbox{Sect}(\omega)\}
\]
is called the sectoriality angle of $A$.
\end{defn}
\begin{remark}
Note that we do not require $A$ to be densely defined.
\end{remark}
The set of sectorial operators on $X$ will be denoted by $\Sect$.

It is well-known that a linear operator $A$ on  $X$ is sectorial
if and only if $(-\infty,0) \subset \rho(A)$ and
\begin{equation}
 M(A):=\sup_{s>0}\,s\|(s+A)^{-1}\|<\infty, \label{stiloper}
\end{equation}
see e.g. \cite[Prop. 1.2.1]{Mart1}.
The notation $M(A)$ introduced above will be important in the sequel.

The holomorphic functional calculus for sectorial operators is
described in details e.g. in \cite{Haa2006} and in \cite{Weis},
and we give only a very short account of it here.

For $\omega \in (0,\pi]$, let
\[
H_0^\infty(\Sigma_\omega):= \left\{f\in \mathcal{O}(\Sigma_\omega)
: \text{$|f(\lambda)|\leq C\min(|\lambda|^s, |\lambda|^{-s})$ for
some $C,s>0$}\right\},
\]
where $\mathcal{O}(S_\omega)$ denotes the space of all holomorphic
functions on the sector $\Sigma_\omega$. Let
\begin{equation}\label{tau}
\tau(\lambda):=\frac{\lambda}{(1+\lambda)^2}
\end{equation}
and
\begin{eqnarray*}
\mathcal{B}(\Sigma_\omega) = \left\{f\in
\mathcal{O}(\Sigma_\omega) : \text{$|f(\lambda)|\leq C\max
\left(|\lambda|^s, |\lambda|^{-s} \right)$ for some
$C,s>0$}\right\}.
\end{eqnarray*}

Let  $A\in \operatorname{Sect}(\theta)$ for some angle $\theta$,
and let $\theta<\omega<\pi$.   Following the abstract scheme
described in subsection \ref{abstrfun}, define
\[
 N=H_0^\infty(\Sigma_\omega),\quad
 M =\mathcal{O}(\Sigma_\omega).
\]
For $f \in H_0^\infty(\Sigma_\omega)$, we define
\begin{equation}\label{Cauchy}
\Phi(f) = f(A):=\frac{1}{2\pi i}\int_\Gamma f(\lambda)
(\lambda-A)^{-1}\,d\lambda,
\end{equation}
where $\Gamma$ is the downward oriented boundary of a sector
$\Sigma_{\omega_0}$, with $\theta<\omega_0<\omega$.  This
definition is independent of $\omega_0$, and
\[
\Phi: H_0^\infty(\Sigma_\omega)\mapsto \mathcal{L}(X),\quad
\Phi(f)=f(A),
\]
is an algebra homomorphism (such that $\Phi(\tau)=A(1+A)^{-2}$).
The mapping $\Phi$ defines the primary holomorphic functional
calculus $(O(\Sigma_\omega), H_0(\Sigma_\omega), \Phi).$

Now let us assume in addition that  $A$ is {\it injective}, so
that $\Phi(\tau)=\tau(A)$ is injective too. Then we can define the
corresponding extended functional calculus
$(O(\Sigma_\omega),\mathcal{B}(\Sigma_\omega), \Phi_e)$ for $A$.
This calculus is called {\em the extended holomorphic calculus}
for $A$.  Any function $f \in \mathcal{B}(\Sigma_\omega)$ has a
regulariser $e$ of the form $e=\tau^n$, thus
\begin{equation}\label{defcal}
f(A)=[\tau^n(A)]^{-1}(\tau^n f)(A),
\end{equation}
where $n\in\mathbb N$ is so large  that
\[
\tau^n f\in H_0^\infty (\Sigma_\omega).
\]

This functional calculus formally depends on a choice of $\omega$,
but the calculi are consistent for different $\omega$'s if we
identify a function $f$ on $\Sigma_\omega$ with its restriction to
$\Sigma_\gamma$ for $\theta < \gamma < \omega$. We may therefore
make this identification and consider our holomorphic calculus to
be defined on the algebra
\[
\mathcal{B}[\Sigma_\omega]:=\bigcup_{\omega<\gamma<\pi}\,\mathcal{B}(\Sigma_\gamma).
\]

Note that  Proposition \ref{funcal} is clearly holds for the
triple $(\mathcal{B}[\Sigma_\omega], O(\Sigma_\omega), \Phi_e)$
and we have in particular the sum rule and the product rule. These
rules will often be used without a specific reference.

The assumption of injectivity of $A$ is often rather restrictive. To avoid it,
we may consider another subalgebra of  $O(\Sigma_\omega)$ defined as
\begin{eqnarray*}
\mathcal{B}_0(\Sigma_\omega):= \left\{f\in
\mathcal{O}(\Sigma_\omega) : \text{$|f(\lambda)|\leq C
|\lambda|^s$ for some
$C,s>0$}\right\}.
\end{eqnarray*}
Every  function $f \in \mathcal{B}_0(\Sigma_\omega)$ has a
regulariser $e$ of the form $e=\tau_0^n$ where $\tau_0(z)=1/(1+z)$ and $n\in \mathbb N$ is large enough. Thus using \eqref{defcal} as above,  we can define the
 extended functional calculus
$(O(\Sigma_\omega),{\mathcal B}_0[\Sigma_\omega], \Phi_e)$ for arbitrary sectorial $A$.
In this way, all fractional powers $A^q, q >0,$ are well-defined, and this definition will be basic for us in dealing
with fractional powers of sectorial operators.

We will frequently use the following composition rule for the
holomorphic functional calculus, see e.g. \cite[Theorem 3.1.4]{Haa2006}.

\begin{prop}\label{frpow3}
Let $\alpha\in (0,1)$ and
\[
0\le \omega<\beta\le \pi, \qquad \beta'=\alpha\beta< \pi.
\]
Suppose that
\[
f\in \mathcal{B}(\Sigma_{\beta'}).
\]
Let $A\in \Sect(\omega)$ and let $A$ be injective. Then
\[
f(A^\alpha)=(f\circ \lambda^\alpha)(A).
\]
\end{prop}

Finally, we will also need the property of sectoriality of
fractional powers of sectorial operators, see e.g. \cite[Theorem 3.1.2]{Haa2006} for
its proof.
\begin{prop}\label{frpow}
Suppose that $A\in \Sect(\omega), q >0,$ and
$q\omega<\pi$. Then $A^q\in \Sect(q\omega)$.
\end{prop}

\subsection{Hirsch functional calculus}\label{funHi}

The definition of the holomorphic functional calculus for a
sectorial operator by means of regularisers is quite implicit.
On the other hand, it is often useful to have a comparatively simple expression
for a function of a sectorial operator allowing for estimates in  resolvent terms.
To this aim, one may use the Hirsch functional
calculus which provides a definition of complete Bernstein function for a sectorial
operator.

We now define complete Bernstein functions of sectorial operators
following Hirsch and review some of their basic properties needed
in the sequel. Let $A$ be a sectorial operator on $X$. The next
definition was essentially given in \cite[p.\ 255]{HirschInt}, see
also \cite{BaGoTa}.

\begin{defn}\label{defoperbern} Given $f \in \mathcal{CBF}$ with  Stieltjes
representation $(a,b,\sigma)$ (see \eqref{Cbf}), define an
operator $f_0(A): \dom(A) \to X$ by
\begin{equation}\label{defbernop}
f_0(A)x=a x + b A x +\int_{0+}^{\infty}A (\lambda +A)^{-1}x\,
\ud\sigma(\lambda), \qquad x \in \dom(A).
\end{equation}
By \eqref{mmu}, this integral is absolutely convergent and
$f_0(A)(I+A)^{-1}$ is a bounded operator on $X$, extending
$(I+A)^{-1}f_0(A)$.  Hence $f_0(A)$ is closable as an operator on
$X$.  Define
\begin{equation}\label{debernop}
f(A)=\overline {f_0(A)}.
\end{equation}
We call $f(A)$ a {\it complete Bernstein function of $A$.}
\end{defn}

Note that by Definition \ref{defoperbern}, $\dom (A)$ is core for
$f(A)$.

The mapping $f \mapsto f(A)$ defined by \eqref{defbernop} posses a
number of properties of functional calculus, see e.g. \cite[Th\'eor\`eme 1-3]{HirschInt} and
\cite[Th\'eor\`eme 1]{HirschFA} (see also \cite[Th. 2.3]{Mart},
\cite[p. 200-201]{HirschFA} and \cite{Hir74}). Thus we call this mapping the {\it Hirsch functional calculus} and  describe some
of its properties below.  Note that the
product rule holds for the Hirsch functional calculus restrictedly since the set of complete
Bernstein functions is not closed under multiplication.
\begin{thm}\label{frpowH}
Let $A$ be a sectorial operator on  $X,$ and let $f$ and $g$ be
complete Bernstein functions. Then the following statements hold.
\begin{enumerate}
\item [\rm (i)]\label{hiri} The operator $f(A)$ (and $g(A)$) is
sectorial and
\begin{equation}\label{frpow1}
\sup_{s>0}\,\|s(f(A)+s)^{-1}\|\le \sup_{s>0}\,\|s(A+s)^{-1}\|.
\end{equation}
\item [\rm (ii)] \label{hirii} The composition rule holds: $f (g (A))= (f\circ g)(A).$
In particular, if $f_\alpha(\lambda)=f(\lambda^\alpha)$ for some
$\alpha\in (0,1)$, then
\begin{equation}\label{AB1}
f(A^\alpha)=f_\alpha(A),
\end{equation}
and if $g_\beta(\lambda):=[g(\lambda)]^\beta\in \mathcal{CBF}$ for
some $\beta>1$, then
\begin{equation}\label{AB2}
[g_\beta(A)]^{1/\beta}=g(A).
\end{equation}
\item [\rm(iii)] \label{hiriii} If the product $f g$ is also a complete Bernstein function, then $f(A)g(A)= (f g)(A).$
\end{enumerate}
\end{thm}

If $\alpha \in  (0,1)$, then remark
after Theorem \ref{Shill} show that $z^\alpha$ is a complete
Bernstein function, and we write $A^\alpha$ for the corresponding
complete Bernstein function of $A$. These fractional powers
coincide with the standard fractional powers defined in the framework of
the holomorphic functional calculus in the previous subsection, see e.g. \cite{BaGoTa}.

\subsection{Hille-Phillips functional calculus}\label{HPfun}

In the functional calculi descri\-bed in Section \ref{funsec}
function of a sectorial operator had to be analytic
on the spectrum of the operator. On the other hand, in some cases
and in the present considerations as well it is possible to drop
this assumption to some extent. In particular, if $-A$ generates a
bounded $C_0$-semigroup then $A \in \Sect(\pi/2),$ where $\pi/2$
cannot be replaced by a smaller angle, in general. On the other
hand, if $f \in L^1(\mathbb R_+)$ then its Laplace transform $\hat
f$ is holomorphic only in $\Sigma_{\pi/2},$ where $\pi/2$ cannot,
in general, be replaced by a larger angle.
 Nevertheless, it is still possible to define
$\hat f (A)$  by the Hille-Phillips functional calculus described
below.

Let us recall definition and basic properties of the (extended)
Hille-Phillips functional calculus useful for the sequel.

Let ${\rm M_b}(\mathbb R_+)$ be a Banach algebra of bounded Radon
measures on $\mathbb R_+.$ If
\[
{\rm A}^1_+(\mathbb C_+) := \{ \widehat \mu : \mu \in {\rm
M}_b(\mathbb R_+)\}
\]
then ${\rm A}^1_+(\mathbb C_+)$ is a commutative Banach algebra
with pointwise multiplication and with the norm
\begin{equation}\label{mmm}
\|\rm \hat \mu\|_{{\rm A}^1_+(\mathbb C_+)} := \|\mu\|_{{\rm
M_b}(\mathbb R_+)} = |{\mu}|(\mathbb R_+),
\end{equation}
where $|\mu|(\mathbb R_+)$ stands for the total variation of $\mu$
on $\mathbb R_+.$

Let $-A$ be the generator of a bounded $C_0$-semigroup
$(e^{-tA})_{t\ge 0}$ on $X$. Then the mapping
\[
{\rm A}^1_+(\mathbb C_+) \mapsto {\mathcal L}(X),\quad
 \Phi (\widehat {\mu})x := \int_0^{\infty} e^{-sA}x\, \mu(d{s}), \qquad x \in X,
\]
defines a continuous algebra homomorphism $\Phi$ such that
\begin{equation}\label{hillestimate}
 \| \Phi (\hat {\mu})\| \le \sup_{t \ge 0} \|e^{-tA}\| |\mu|(\mathbb R_+).
\end{equation}
The homomorphism $\Phi$ is called the (primary) {\em
Hille-Phillips} (HP-) functional calculus for $A$, and if
$g=\hat \mu$ then we set
\begin{equation*}
g(A)= \Phi(\widehat {\mu}).
\end{equation*}
Basic properties of the Hille-Phillips functional calculus can be
found in  \cite[Chapter XV]{HilPhi}.

As in the case of holomorphic functional calculus, we can define
the extended $HP$-calculus by the extension procedure described in Subsection \ref{abstrfun}.

 Let $O(\mathbb C_+)$ be an algebra of functions holomorphic in $\mathbb C_+.$
 If $f \in O(\mathbb C_+)$ is holomorphic such that there exists
 $e\in {\rm A}^1_+(\mathbb C_+)$ with $ef \in {\rm A}^1_+(\mathbb C_+)$ and the
operator $e(A)$ is injective, then one defines $f(A)$ as in
\eqref{abstrdef}.
We then call the triple $(O(\mathbb C_+), {\rm A}^1_+(\mathbb C_+),
\Phi_e)$
 the {\em extended Hille--Phillips calculus} for $A$.

Using the extended Hille-Phillips functional calculus, one can
define a Bernstein function $\psi(A)$ of $A$ and obtain a useful,
L\'evy-Hintchine type representation  for $\psi(A).$ Recall  first
that Bernstein functions are regularisable by $e(z)=1/(1+z).$

\begin{prop}\label{PrS} \cite[Lemma 2.5]{GHT}
Let $\psi\in \mathcal{BF}$. Then $\psi(z)/(1+z)\in \Wip(\C_+)$.
\end{prop}

Proposition \ref{PrS} implies that for $\psi\in \mathcal{BF}$ the
operator $\psi(A)$ can  be defined by the extended $HP$-calculus
as
\begin{equation}\label{part}
\psi(A)=(1+A)[\psi(z) (1+z)^{-1}](A)
\end{equation}
with the natural domain
\[
\dom(\psi(A))=\{x\in X:\,[\psi(z)(1+z)^{-1}](A)x\in \dom(A)\}.
\]

As it was proved in \cite[Corollary 2.6]{GHT}, the formula \eqref{part}
can be written in a more explicit and useful form called the
L\'evy-Hintchine representation. Moreover, Bernstein functions of
semigroup generators possess certain permanence properties.
Namely, for any $\psi\in \mathcal{BF}$ the operator $-\psi(A)$
generates a bounded $C_0$-semigroup on $X$ as well and the latter
semigroup can be represented in terms of $\psi$ and $(e^{-tA})_{t
\ge 0}$ as (\ref{CMonG})  suggests, see e.g. \cite[Proposition 13.1
and Theorem 13.6] {SchilSonVon2010} and \cite{Phil52}.

For later reference, we summarize these results below.
\begin{thm}\label{BochP}
\begin{itemize}
\item [(i)] Let $-A$ generate a bounded $C_0$-semigroup $(e^{-tA})_{t
\ge 0}$ on  $X$, and let $\psi$ be a Bernstein function with the
corresponding Levy-Hintchine representation $(a, b,\mu).$ Then
 $\left.\psi(A)\right|_{{\rm dom}(A)}$ is given by
 \begin{equation}\label{clos}
\psi(A)x = ax + bAx + \int_{0+}^\infty (1 - e^{-sA})x \,
\mu(d{s}),\qquad x\in \dom(A),
\end{equation}
where the integral is understood as a Bochner integral, and ${\dom}(A)$ is  core for $\psi (A).$
\item [(ii)] Moreover, $-\psi(A)$ is the generator of a bounded $C_0$-semigroup $(e^{-t\psi(A)})_{t \ge 0}$
on $X$ given by
\[
e^{-t\psi(A)}:=\int_{0}^{\infty} e^{-sA}\, \mu_t(ds), \qquad t \ge
0,
\]
where $(\mu_t)_{t\ge 0}$ is a vaguely continuous convolution
semigroup   of subprobability measures on $[0,\infty)$
corresponding to $\psi$ by \eqref{vague} (cf. \eqref{CMonG}).
\end{itemize}
\end{thm}

The semigroup $(e^{-t\psi(A)})_{t\ge 0}$ defined in Theorem
\ref{BochP} is called {\em subordinate} to the semigroup
$(e^{-tA})_{t\ge 0}$ with respect to the Bernstein function
$\psi$. Theorem \ref{BochP}, (ii), implies that $(-\infty,0)\subset \rho(\psi(A))$ and
\begin{equation}\label{SupG}
\sup_{s>0}\|s(s+\psi(A))^{-1}\|\le \sup_{t>0}\,\|e^{-t\psi(A)}\|\le
\sup_{t>0}\,\|e^{-t A}\|.
\end{equation}

The next approximation result will often allow us to reduce
considerations to the case when the semigroup generator  is
invertible.

\begin{prop}\label{limit}
Let $A$ be the generator of a bounded $C_0$-semigroup on $X$ and
let $\psi$ be  a Bernstein function.
\begin{itemize}
\item [(i)]
If $\epsilon >0$ then $[\psi(\cdot+\epsilon)-\psi(\cdot)](A) \in \mathcal{L}(X)$
so that $\dom (\psi(A+\epsilon))=\dom (\psi(A))$ and
\begin{equation}\label{sumphil}
\psi(A+\epsilon)= \psi(A) + [\psi(\cdot +\epsilon)-\psi(\cdot)](A).
\end{equation}
\item [(ii)]
If  $x \in \dom(A)$ and $\epsilon>0,$ then
\begin{equation}
 \|\psi(A+\epsilon)x-\psi(A)x\| \le M (\psi(\epsilon)-\psi(0))\|x\|.
\end{equation}
where $M:=\sup_{t \ge 0}\|e^{-tA}\|\ge 1.$
\item [(iii)] For every $s>0,$
\begin{equation}\label{diffresolv}
\lim_{\epsilon \to
0+}\|(s+\psi(A+\epsilon))^{-1}x - (s+\psi(A))^{-1}x\|=0, \qquad x \in X.
\end{equation}
\end{itemize}
\end{prop}
\begin{proof}
To prove (i) and (ii), we note  that if $\psi$ is a Bernstein function with the
L\'evy-Hintchine representation $(a,b,\mu)$ then
for all  $\epsilon>0$ and $\lambda \in \mathbb C_+,$
\[
\psi(\lambda+\epsilon)-\psi(\lambda) =b\epsilon+\int_{0+}^\infty
e^{-\lambda s}(1-e^{-\epsilon s})\,\mu(ds).
\]
Hence, $\psi(\cdot +\epsilon)-\psi(\cdot) \in A^1_+(\mathbb C_+)$ and
$$
\| \psi(\cdot +\epsilon)-\psi(\cdot)\|_{ A^1_+}=\psi(\epsilon)-\psi(0).
$$
From here by the (extended) HP-calculus it follows that
$$
\|[\psi(\cdot +\epsilon)-\psi(\cdot)](A)\| \le M (\psi(\epsilon)-\psi(0)).
$$
Moreover, since $[\psi(\cdot +\epsilon)-\psi(\cdot)](A) \in \mathcal{L}(X),$ by the sum rule
for the (extended) HP-calculus
we obtain that $\dom (\psi(A+\epsilon))=\dom (\psi(A))$ and \eqref{sumphil} holds.
Since $\dom (A) \subset \dom (\psi(A))$
by Theorem \ref{BochP}, (i), we have also
\begin{eqnarray}\label{Psi}
\|\psi(A+\epsilon)x-\psi(A)x\| \le M[\psi(\epsilon)-\psi(0)]\|x\|,
\end{eqnarray}
for all $x \in \dom (A)$. This finishes the proof of (i) and (ii).

Furthermore, by the product rule for the (extended) HP-calculus, for all $\epsilon, s >0$ and $x \in \dom (A),$
\begin{equation}
(s+\psi(A))^{-1}(\psi(A+\epsilon)-\psi(A))x=(\psi(A+\epsilon)-\psi(A))(s+\psi(A))^{-1}x.
\end{equation}
Hence the estimates \eqref{SupG} and \eqref{Psi} yield
\begin{eqnarray}\label{Psi4}
&&\|(s+\psi(A))^{-1}x-(s+\psi(A+\epsilon))^{-1}x\|\\
&=&\|(s+\psi(A))^{-1}(s+\psi(A+\epsilon))^{-1}(\psi(A+\epsilon)-\psi(A))x\|
\notag \\
&\le&  \frac{M^2}{s^2}\|\psi(A+\epsilon)x-\psi(A)x\|\notag \\
&\le&
\frac{M^3}{s^2}(\psi(\epsilon)-\psi(0))\|x\|\notag \\
&\to 0,& \quad \epsilon\to 0+.\notag
\end{eqnarray}
As $\dom(A)$ is dense in $X,$ (iii) follows.
\end{proof}

Let $\psi\in \mathcal{BF}$. Then $\psi$ is holomorphic in $\C_{+}$
and continuous in $\overline{\C}_{+}$, and by Lemma \ref{psi}, (ii)
 we have
\begin{equation}\label{RR1}
\tau^2\psi\in H_0^\infty(\C_{+}) \quad \text{and} \quad \tau^2/\psi\in
H_0^\infty(\Sigma_\omega),\quad \omega\in (0,\pi/2),
\end{equation}
where $\tau$ is defined by (\ref{tau}).

Let $A\in \Sect(\alpha)$ for some $\alpha\in [0,\pi/2),$ so that,
in particular, $-A$ generates a sectorially bounded holomorphic
$C_0$-semigroup. Suppose that $A$ has dense range. Let $f \in \mathcal P,$ that is  $f:=1/\psi$ for some  $\psi\in
\mathcal{BF}, \psi \not \equiv 0.$  Then, by (\ref{RR1}), the operators
$\psi(A)$ and $f(A)$ can be defined by the (extended) holomorphic
calculus with the regulariser $\tau_\epsilon$ of the form
\begin{equation}\label{part1}
\tau_\epsilon(z):=\left(\frac{z}{(\epsilon+z)(1+\epsilon z)}\right)^2,
\end{equation}
for any fixed $\epsilon>0.$ Thus by \cite[Proposition
1.2.2,\,d]{Haa2006}, we have
\begin{equation}\label{fpsi}
f(A)=[\psi(A)]^{-1}.
\end{equation}
Moreover, for every  $h$ of the form
\begin{equation}\label{sum}
h=\psi+f,\qquad \psi\in \mathcal{BF},\quad f\in \mathcal{P},
\end{equation}
and every $A\in {\rm \Sect}(\alpha)$, $\alpha\in [0,\pi/2)$, with
dense range, the (closed) operator $h(A)$ is well-defined by the
extended holomorphic functional calculus (with the
regulariser $\tau_\epsilon$, $\epsilon>0$.)

\begin{prop}\label{regSP}
Let $h$ be of the form (\ref{sum}) and $A\in \Sect(\alpha)$,
$\alpha\in [0,\pi/2)$. If $A$ has dense range then
\begin{equation}\label{Hh}
\overline{\psi(A)+f(A)}=h(A).
\end{equation}
\end{prop}

\begin{proof}
We have
\[
\lim_{\epsilon \to 0}\,\tau_\epsilon(A)x=x,\qquad x\in X,
\]
and then the statement follows from  \cite[Proposition 3.3]{Ha05}
or \cite[Proposition 3.2]{Clark}.
\end{proof}

\subsection{Compatibility of functional calculi}\label{compat}

This subsection allows us to unify the results of previous
subsections and to show that different definitions of the same
function of semigroup generator are compatible.

The first result proved in \cite[Theorem 4.12]{BaGoTa}  shows that
the extended holomorphic functional calculus and the Hirsch
functional calculus are compatible.
\begin{prop}\label{Sovp2}
Let $\psi\in \mathcal{CBF}$ and let $A$ be an injective sectorial
operator on $X$. Then the operator $\psi(A)$ defined by the Hirsch
calculus coincides with  $\psi(A)$ defined via the extended
holomorphic functional calculus.
\end{prop}

The second result yields compatibility of the extended
Hille-Phillips calculus and the extended holomorphic calculus on
appropriate functions.
\begin{prop}\label{Sovp1}
Let $\psi\in \mathcal{BF}.$ Suppose that $\psi$ admits holomorphic
extension $\tilde{\psi}$ to $\Sigma_\omega$ for some
$\omega\in (\pi/2,\pi)$ so that
\[
\tilde{\psi}\in \mathcal{B}(\Sigma_\omega).
\]
Let $-A$ be the generator of a bounded $C_0$-semigroup and let $A$
be injective. Let $\psi(A)$ be defined by the extended
Hille-Phillips calculus and $\tilde{\psi}(A)$ be defined via the
extended holomorphic functional calculus. Then
\begin{equation}\label{that}
\tilde{\psi}(A)=\psi(A).
\end{equation}
\end{prop}

\begin{proof}
Note that for some $n\in \N$
\[
\tilde{\tau}(\lambda)=\tau^n(\lambda)=\left(\frac{\lambda}{(\lambda+1)^2}\right)^n,
\]
is a  regulariser for $\tilde{\psi}$ in the holomorphic functional
calculus. By  Proposition \ref{PrS},
the function $\tilde{\tau}$ is a
regulariser for $\psi$ in the extended $HP$-calculus, i.e. there exists a  bounded Radon
measure $\mu$ on $\R_{+}$ such that
\[
\tilde{\tau} \in \Wip(\C_+),\quad (\tilde{\tau} \psi)(\lambda)=\int_0^\infty
e^{-\lambda s}\,\mu(ds)\in \Wip(\C_+).
\]
Then, by \cite[Proposition 3.3.2]{Haa2006} on compatibility of the
Hille-Phillips and the holomorphic functional calculi, we have
\begin{equation}\label{same}
(\tilde{\tau} \tilde{\psi})(A)=\int_0^\infty e^{-sA}\,\mu(ds)=(\tilde{\tau}\psi)(A).
\end{equation}
By the same \cite[Proposition 3.3.2]{Haa2006}, the operators
$\tilde{\tau}(A)$ given by the holomorphic calculus and by the
$HP$-calculus are the same. So, (\ref{same}) implies
(\ref{that}).
\end{proof}

\section{Resolvent estimates for certain functions of sectorial operators}\label{frac}

The following estimates of independent interest will be
instrumental in subsequent sections. In a qualitative form, they are essentially known.
However, for our purposes, it is crucial to equip them with explicit constants and specify large enough sectors 
where such estimates hold.

The first estimate is a version of the well-known
result on sectoriality of fractional powers due to Kato,
\cite[Theorem 2]{Kato1960}. However we give an explicit constant in the
sectoriality condition restricted to an appropriate sector, and
this could be helpful in many instances.

\begin{prop}\label{PrBBL1}
Let $A\in {\rm Sect},$ so that
\begin{equation*}
\|s(A+s)^{-1}\|\le M,\qquad s>0.
\end{equation*}
If $r\in (0,1)$ and $\gamma \in (0,(1-r)\pi),$ then
\begin{equation}\label{Zzz}
\|(A^r+z)^{-1}\|\le \frac{M\sin(\pi r)}{\pi r} \frac{(\pi
r+\gamma)}{\sin(\pi r+\gamma)}\frac{1}{|z|}, \quad z\in
\Sigma_\gamma.
\end{equation}
\end{prop}

\begin{proof}
If $z\in \Sigma_{\gamma},$ then by \cite[Theorem 2]{Kato1960} we have
\begin{eqnarray}\label{Kr10}
(A^r+z)^{-1}&=& \frac{\sin(\pi r)}{\pi} \int_0^\infty\frac{t^r
(A+t)^{-1}\,dt}
{(t^re^{i\pi r}+z)(t^r e^{-i\pi r}+z)}\\
&=&\frac{\sin(\pi r)}{\pi} \int_0^\infty\frac{t^r (A+t)^{-1}\,dt}
{t^{2r}+2t^r z\cos \pi r+z^2}.\notag
\end{eqnarray}
Since
\[
|t^re^{\pm i\pi r}+z|\ge |t^re^{i(\pi r+\gamma)}+|z||,\quad
t>0,\qquad z\in \Sigma_\gamma,
\]
and
\[
\int_0^\infty \frac{dt}{|te^{i\beta}+s|^2}=\frac{\beta}{s\sin \beta},
\quad \beta\in (-\pi,\pi),\quad s>0
\]
(see e.g. (\cite[Formula 2.2.9.25]{Prudnikov}), (\ref{Kr10})
implies that
\begin{eqnarray*}
\|(A^r+z)^{-1}\|&\le& \frac{M\sin(\pi r)}{\pi}
\int_0^\infty\frac{t^{r-1}\,dt} {|t^re^{i\pi r}+z| |t^r e^{-i\pi
r}+z|}
\\
&\le& \frac{M\sin(\pi r)}{\pi} \int_0^\infty\frac{t^{r-1}\,dt}
{|t^re^{i(\pi r+\gamma)}+|z||^2}\\
&=& \frac{M\sin(\pi r)}{\pi r} \int_0^\infty\frac{dt}
{|te^{i(\pi r+\gamma)}+|z||^2}\\
&=&\frac{M\sin(\pi r)}{\pi r} \frac{\pi r+\gamma}{\sin(\pi
r+\gamma)}\frac{1}{|z|}, \qquad z\in \Sigma_\gamma.
\end{eqnarray*}
\end{proof}

The next result is an explicit version of Proposition \ref{frpow}
for $q>1$. We are
not aware of an estimate of this kind in the literature.

\begin{prop}\label{PrBBL10}
Let operator $A\in \Sect(\alpha)$, $\alpha\in (0,\pi),$ so that
\begin{equation}\label{EsR0}
\|z(A+z)^{-1}\|\le M(A,\alpha'),\qquad z\in
\Sigma_{\pi-\alpha'},\quad \alpha'\in (\alpha,\pi),
\end{equation}
and, in particular,
\begin{equation}\label{part0}
\|s(A+s)^{-1}\|\le M(A),\qquad s>0.
\end{equation}
Let $q>1$ be such that $q\alpha<\pi$. Then $A^q\in \Sect(q\alpha)$
and, moreover, for every $\gamma \in (0,\pi-q\alpha),$
\begin{equation}\label{MEs0}
\|z(A^q+z)^{-1}\|\le \tilde{M}_{\alpha,q;\gamma},\qquad z \in \Sigma_{\gamma},
\end{equation}
where
\begin{equation}\label{Mn}
\tilde{M}_{\alpha,q;\gamma}:=
M(A)+\frac{2M(A,\beta_{\alpha,q;\gamma})}{\pi\cos(\beta_{\alpha,q;\gamma}/2)
\cos (q\beta_{\alpha,q;\gamma}/2)}, \qquad
\beta_{\alpha,q;\gamma}:=\frac{\alpha+(\pi-\gamma)/q}{2}.
\end{equation}
\end{prop}

\begin{proof}
Let $\gamma \in (0,\pi-q\alpha)$ be fixed.
Since $f_{s,q} \in H_{0}^{\infty}(\Sigma_\pi),$ using the holomorphic functional calculus and \cite[Proposition 3.1.2]{Haa2006}, we have
\[
z (A^q+z)^{-1}=|z|^{1/q}(|z|^{1/q}+A)^{-1}+f_{z,q}(A),\qquad z\in
\Sigma_{\pi-q\alpha},
\]
where
\[
f_{z,q}(\lambda):=\frac{z}{z+\lambda^q}-\frac{|z|^{1/q}}{\lambda+|z|^{1/q}}=
\frac{z\lambda -z^{1/q}\lambda^q}{(z+\lambda^q)(\lambda+|z|^{1/q})}.
\]
Hence,
\[
\|z (A^q+z)^{-1}\|\le M(A)+\|f_{z,q}(A)\|,\quad
z \in\Sigma_{\pi-q\alpha},
\]

Furthermore, if $\beta\in (\alpha,\pi/q)$ and $z \in \Sigma_\gamma,$  then
\begin{eqnarray*}
\|f_{z,q}(A)\|&\le& \frac{1}{2\pi}\int_{\partial S_{\beta}}
|f_{z,q}(\lambda)|\|(\lambda-A)^{-1}\|\,|d\lambda|\\
&\le& \frac{M(A,\beta)}{2\pi}\int_{\partial S_{\beta}}
|f_{z,q}(\lambda)|\frac{|d\lambda|}{|\lambda|}, \quad \beta\in (\alpha,\pi/q).
\end{eqnarray*}
Moreover, if,  in addition, $q\in (1,\pi/\alpha)$ and $z \in \Sigma_\gamma,$ then setting
$C=\cos(\beta/2)\cos((q\beta+\gamma)/2)$ and  using
(\ref{number1}), we have
\begin{eqnarray*}
\int_{\partial S_{\beta}} |f_{z,q}(\lambda)|\frac{|d\lambda|}{|\lambda|} &\le&
\int_{\partial S_{\beta}}
\frac{|z|+|z|^{1/q}|\lambda|^{q-1}}{|z+\lambda^q|
|\lambda+|z|^{1/q}|} |d\lambda|\\
&\le C&\int_{\partial S_{\beta}}
\frac{|z|+|z|^{1/q}|\lambda|^{q-1}}{(|z|+|\lambda|^q|)
(|\lambda|+|z|^{1/q})} |d\lambda|
\\
&=&2C\int_0^\infty\frac{1+t^{q-1}}{(t^q+1) (t+1)}\,dt.
\end{eqnarray*}
Since $q \ge 1,$
\[
\int_0^\infty \frac{(1+t^{q-1})\,dt}{(t^q+1)(t+1)}\le
2\int_0^\infty \frac{dt}{(t+1)^2}=2,
\]
thus
\[
\|z (A^q+z)^{-1}\|\le M(A)+\frac{2
M(A,\beta)}{\pi\cos(\beta/2)\cos ((q\beta+\gamma)/2)}, \qquad
z \in\Sigma_\gamma.
\]
Putting  $\beta=\beta_{\alpha,q;\gamma},$ we obtain (\ref{MEs0}) and  (\ref{Mn}).
\end{proof}

In \cite[Theorem 6.1 and Remark 6.2]{Laub} it was proved that if
$\psi\in\mathcal{CBF}$ then
\begin{equation}\label{BLimp}
A\in \Sect(\theta) \Longrightarrow \psi(A)\in \Sect(\theta),\qquad
\theta\in [0,\pi/2).
\end{equation}
The proof of \eqref{BLimp} there
was based on the
fact that
\begin{equation}\label{Stat}
\psi\in\mathcal{CBF} \Longrightarrow [\psi(\lambda^\alpha)]^{1/\alpha}\in \mathcal{CBF}, \qquad
\alpha\in (0,1),
\end{equation}
and on Theorem \ref{frpowH}. (See also \cite[Corollary 2]{Nakamura} and
\cite[Corollary 7.15] {SchilSonVon2010}.) We
present below a slight generalization of (\ref{BLimp}) extending
it to the whole class of sectorial operators. Once again, apart
from proving the sectoriality of $\psi (A)$, we give explicit
constants in the resolvent bounds for $\psi(A)$ and get a control
over the sectoriality angle of $\psi(A).$ This will be used in
subsequent sections.

\begin{thm}\label{LB}
Let $A\in \Sect(\alpha)$, $\alpha\in (0,\pi)$. If $\psi\in
\mathcal{CBF},$ then $\psi(A)\in \Sect(\alpha)$ too. Moreover, if
$q \in (1,\pi/\alpha)$  and $\gamma \in (0,(1-q^{-1})\pi)$,
then
\begin{equation}\label{Aest}
\|z(\psi(A)+z)^{-1}\| \le \frac{\tilde{M}_{\alpha,q;\gamma}(A)\sin(\pi/q)}{\pi/q}
\frac{\pi/q +\gamma}{\sin(\pi/q+\gamma)}, \quad z\in
\Sigma_\gamma,
\end{equation}
where
 $\tilde{M}_{\alpha,q;\gamma}(A)$ is given by
(\ref{Mn}).
\end{thm}

\begin{proof}
In this proof, we will combine the (extended) holomorphic functional calculus
and the Hirsch functional calculus. This is possible due to compatibility of the calculi given by Proposition
\ref{Sovp2}.

Recall first that by  Theorem \ref{frpowH}, (i), the operator $\psi(A)$ is sectorial.
Choose $q\in (1,\pi/\alpha)$ so that  $1/q\in (0,1),$ and define
$g_{1/q}(z):=z^{1/q}$ and $\eta_q(z):=z^q$. Since $g_q\circ \psi \circ
g_{1/q}\in \mathcal{CBF}$ (see  (\ref{Stat})) and $A^q$ is
sectorial in view of Proposition \ref{frpow}, we infer by
(\ref{AB2}) that
\[
[(g_q\circ \psi \circ g_{1/q})(A^q)]^{1/q}=[\psi\circ
g_{1/q}](A^q),
\]
and, moreover, by (\ref{AB1}),
\[
[\psi\circ g_{1/q}](A^q)=\psi((A^q)^{1/q})=\psi(A).
\]
Hence
\begin{equation}\label{comp}
 \varphi:=g_q\circ \psi \circ g_{1/q}\in \mathcal{CBF} \qquad \text{and} \qquad \psi(A)=[\varphi(A^q)]^{1/q},
\end{equation}
and, by Proposition \ref{frpow}, we obtain that
\begin{equation}\label{Sect11}
\psi(A)\in \Sect(\pi/q).
\end{equation}
As $q\in (1,\pi/\alpha)$ is arbitrary, choosing
 $q$ closely enough to $\pi/\alpha$
we can make $\gamma$ arbitrarily close to $\pi-\alpha.$
Thus,  from (\ref{Sect11}) it follows that $\psi(A)\in \Sect(\alpha)$.

Let now
$\gamma \in (0,(1-q^{-1})\pi)$
and $z\in \Sigma_\gamma.$  Using (\ref{comp}),
Proposition \ref{PrBBL1} with $r=1/q,$ Theorem \ref{frpowH}, (i),
and Proposition \ref{PrBBL10}, we conclude that
\begin{eqnarray*}
\|z(\psi(A)+z)^{-1}\|&=&\|z([\varphi(A^q)]^{1/q}+z)^{-1}\|\\
&\le&\frac{\sin(\pi/q)}{\pi/q} \frac{(\pi/q+\gamma)}{\sin(\pi/q
+\gamma)} \sup_{s>0}\,\|s(\varphi(A^q)+s)^{-1}\|
\\
&\le& \frac{\sin(\pi/q)}{\pi/q} \frac{(\pi/q +\gamma)}{\sin(\pi/q
+\gamma)}\, \sup_{s>0}\,\|s(A^q+s)^{-1}\|
\\
&\le& \tilde{M}_{\alpha,q;\gamma}(A)\frac{\sin(\pi/q)}{\pi/q} \frac{\pi/q+\gamma}{\sin(\pi/q+\gamma)},
\end{eqnarray*}
where $\tilde{M}_{\alpha,q;\gamma}(A)$ is defined by (\ref{Mn}).

Once again, since the choice of  $q\in (1,\pi/\alpha)$ and $\gamma
\in (0,(1-q^{-1})\pi)$ is arbitrary,
it follows that
 (\ref{Aest}) holds for $z\in \Sigma_{\pi-\alpha'}$ for any  $\alpha'\in (\alpha,\pi)$.
\end{proof}

\section{Main results: holomorphicity and preservation of angle}\label{Main1}

We start with recalling some basics on holomorphic semigroups.

A $C_0$-semigroup $(e^{-tA})_{t \ge 0}$ is said to be a
holomorphic  semigroup of angle $\theta$ if $e^{-\cdot A}$ extends holomorphically
to a sector $\Sigma_\theta $ for some $\theta \in
(0,\frac{\pi}{2}].$ In this case, we write $-A \in \mathcal H(\theta).$
If this extension is bounded in $\Sigma_{\theta'}$ for every $\theta' \in (0,\theta)$
then we say that
 $(e^{-tA})_{t \ge 0}$ is a sectorially bounded holomorphic  semigroup of
angle $\theta$ and write
$-A\in \mathcal{BH}(\theta)$.
Note that  $(e^{-tA})_{t>0}$ may admit a holomorphic extension to
$\Sigma_\theta$ as above without being sectorially bounded (as
already one-dimensional examples show).
Recall that if $\omega\in [0,\pi/2)$ then
 $A\in \Sect(\omega)$ if and only if
$-A\in \mathcal{BH}(\pi/2-\omega)$, see e.g. \cite[Theorem 4.6]{EngNag}.

Berg, Boyadzhiev and de Laubenfels  proved in \cite[Propositions 7.1 and 7.4]{Laub} that if
$-A\in \mathcal{BH}(\theta)$ and $\theta \in (\pi/4,\pi/2]$, then for
any $\psi\in \mathcal{BF}$ the operator $-\psi(A)$ generates a
sectorially bounded holomorphic $C_0$-semigroup, and if  $-A\in
\mathcal{BH}(\pi/2)$, then $-\psi(A)\in \mathcal{BH}(\pi/2)$ too.
They also asked in \cite{Laub} whether the statement holds for $\theta$ from the whole of the interval $(0,\pi/2].$
In Theorem \ref{Tangle} below, we remove the restriction on
$\theta$ and prove the result in full generality thus solving a problem posed in \cite{Laub}. Moreover, we show that $\psi$ the holomorphy angle of
of $(e^{-tA})_{t \ge 0}$ invariant. As byproduct,  in Corollary \ref{Tangle1}, we
also answer the question by Kishimoto-Robinson from \cite{Rob} mentioned in Introduction.
To this aim, we will first need to prove several results on functional calculi allowing one to apply the estimate \eqref{need} proved in  Theorem \ref{BTR}.

Let $-A\in \mathcal{BH}(\theta)$ for some $\theta\in (0,\pi/2]$ so
that for every $\omega\in (0,\pi/2+\theta),$
\begin{equation}\label{sA}
\|(z+A)^{-1}\|\le \frac{M(A,\omega)}{|z|},\qquad z\in
\Sigma_{\omega}.
\end{equation}
The following assumption will be crucial:\\
\emph{ For the whole of this section, let
$\omega\in(\pi/2,\pi/2+\theta)$
be fixed. Let also $\psi$ be a Bernstein function, $\varphi$ be the
complete Bernstein function associated to $\psi,$ and the
function $r$ be given by \eqref{defr}.}

Define $r(A,\cdot):\Sigma_\omega \to \mathcal L(X)$ and $F(A,\cdot):\Sigma_\omega \to \mathcal L(X)$
by
\begin{equation}\label{RRR}
r(A;z):=\frac{1}{2\pi i} \int_{\partial
\Sigma_\beta}\,r(\lambda;z) (\lambda-A)^{-1}\,d\lambda,
\end{equation}
and
\begin{equation}\label{RRF}
F(A;z):=\frac{1}{2\pi i} \int_{\partial\Sigma_\beta} \frac{\lambda
r(\lambda;z)}{(\lambda+1)^2}(\lambda-A)^{-1}\,d\lambda,
\end{equation}
where  $\beta\in (\pi/2-\theta, \pi-\omega)$ is arbitrary and $\Sigma_\beta$ is oriented counterclockwise.
In view of \eqref{need} and Cauchy's theorem, the functions $r$ and $F$ are well-defined.

We start with providing sectoriality estimates for $r$ in appropriate sectors an expressing $F$ via $r.$
\begin{prop}\label{Tr}
Let $-A\in \mathcal{BH}(\theta)$ for some $\theta\in (0,\pi/2]$ so
that (\ref{sA}) holds. Then for every
$\omega\in(\pi/2,\pi/2+\theta),$
$r(A,\cdot)$ is holomorphic in $\Sigma_\omega$ and
for every $\beta\in (\pi/2-\theta,\pi-\omega),$
\begin{equation}\label{sEs}
\|r(A;z)\|\le \frac{4 M(A,\pi-\beta)}{\pi
\cos^2\beta\,\cos^2((\gamma+\beta)/2)\,|z|},\quad z\in \Sigma_\omega.
\end{equation}
In particular, if  $\delta:=\pi/2+\theta-\omega,$ then
\begin{equation}\label{SMS}
\|r(A;z)\|\le  \frac{4 M(A,\pi/2+\theta-\delta/2)}{\pi
\sin^2(\theta/2)\,\sin^2(\delta/4)\,|z|},\qquad z\in
\Sigma_\omega.
\end{equation}
\end{prop}

\begin{proof}
The estimate (\ref{sEs}) follows from (\ref{RRR}), (\ref{need})
and (\ref{sA}).
Setting
\[
\delta=\pi/2+\theta-\omega,\quad \beta =\pi/2-\theta+\delta/2,
\]
in (\ref{sEs}), we obtain  (\ref{SMS}).

The holomorphicity of $r(A, \cdot)$ in $\Sigma_\gamma$ is a direct consequence of Fubini's and Morera's theorems.
\end{proof}

\begin{lemma}\label{Resolvent}
Let $r(A;z)$ and $F(A;z)$ be defined by (\ref{RRR}) and
(\ref{RRF}), respectively. Then
\begin{equation}\label{SA}
F(A;z)=A(A+1)^{-2}r(A;z),\qquad z\in \Sigma_\omega.
\end{equation}
\end{lemma}

\begin{proof}
Note that
\begin{eqnarray*}
\frac{\lambda}{(\lambda+1)^2}-A(A+1)^{-2}&=&
[\lambda(A+1)^2-(\lambda+1)^2A]\frac{(A+1)^{-2}}{(\lambda+1)^2}
\\
&=&(\lambda A-1)(A-\lambda)\frac{(A+1)^{-2}}{(\lambda+1)^2}, \qquad \lambda \in \mathbb C\setminus (-\infty,0).
\end{eqnarray*}
Therefore, by (\ref{SS2}), for every  $x\in X$ one has
\begin{eqnarray*}
&&F(A;z)x-A(A+1)^{-2}r(A;z)x\\
 &=&\frac{1}{2\pi i}
\int_{\partial\Sigma_\beta} r(\lambda;z)\left[
\frac{\lambda}{(1+\lambda)^2} -A(A+1)^{-2}\right]
(\lambda-A)^{-1}x\,d\lambda
\\
&=&-\frac{1}{2\pi i} \int_{\partial\Sigma_\beta}
\frac{r(\lambda;z)}{(\lambda+1)^2} (\lambda A-1)(A+1)^{-2}x
\,d\lambda
\\
&=&\frac{1}{2\pi i} \int_{\partial\Sigma_\beta}
\frac{r(\lambda;z)}{(\lambda+1)^2} (A+1)^{-2}x \,d\lambda\\
&-&\frac{1}{2\pi i} \int_{\partial\Sigma_\beta} \frac{\lambda
r(\lambda;z)}{(\lambda+1)^2} A (A+1)^{-2}x
\,d\lambda\\
&=&0.
\end{eqnarray*}
\end{proof}

The following statement relating  the resolvents of $\psi(A)$ and
$\varphi(A)$ will be  basic for proving the main result of this
paper, Theorem \ref{Tangle}. It shows that the resolvents do not
differ much as far their behavior at infinity is concerned.

\begin{prop}\label{SC1}
Let $-A\in \mathcal{BH}(\theta)$ for some $\theta\in (0,\pi/2]$.
Then
\begin{equation}\label{SA1}
(z+\psi(A))^{-1}=(z+\varphi(A))^{-1}+r(A;z),\quad z\in
\Sigma_\omega.
\end{equation}
\end{prop}

\begin{proof}
Suppose first that $A$ has dense range. Then the operators $
(s+\psi)^{-1}(A) $ and $(s+\varphi)^{-1}(A)$ are
well-defined for $s>0$ via the (extended) holomorphic functional calculus
with the regulariser $\tau$ given  by (\ref{tau}). On the other
hand, since $-\psi(A)$ and $-\varphi(A)$ generate bounded $C_0$-semigroups, we
have that
\[
(s+\psi(A))^{-1}\in \mathcal{L}(X)  \quad \text{and} \quad (s + \varphi(A))^{-1} \in\mathcal{L}(X),\qquad s>0.
\]
Moreover, by \cite[Theorem 1.3.2, $f)$]{Haa2006}, if  $s>0$, then
\[
(s+\psi(A))^{-1}=(s+\psi)^{-1}(A) \quad \text{and} \quad
(s+\varphi(A))^{-1}=(s+\varphi)^{-1}(A).
\]
Hence, using sum rule for the (extended) holomorphic functional calculus,
\[
(s+\psi(A))^{-1}-(s+\varphi(A))^{-1}
=[(s+\psi)^{-1}-(s+\varphi)^{-1}](A).
\]
Furthermore, using the holomorphic functional calculus once again,
\begin{eqnarray*}
(s+\psi(A))^{-1}-(s+\varphi(A))^{-1}&=&[\tau(A)]^{-1}[((s+\psi)^{-1}-(s+\varphi)^{-1})\tau](A)\\
&=&[A(A+1)^{-2}]^{-1} [r(s;\cdot)\tau](A)\\
&=&[A(A+1)^{-2}]^{-1}F(A;s).
\end{eqnarray*}
 From this and (\ref{SA}) it follows that
\begin{equation}\label{SA12}
(s+\psi(A))^{-1}=(s+\varphi(A))^{-1}+r(A;s),\qquad s>0,
\end{equation}
that is \eqref{SA1} holds for $z >0.$

To obtain \eqref{SA12} in case when the range of $A$ may not be dense, we consider the
approximation of $A$ by the operators $A_\epsilon$ with dense range given by
\[
A_\epsilon:=A+\epsilon\in \mathcal{BH}(\theta),\quad \epsilon>0.
\]
By (\ref{SA12}) we have
\begin{equation}\label{SA11}
(s+\psi(A_\epsilon))^{-1}-(s+\varphi(A_\epsilon))^{-1}=r(A_\epsilon;s),\qquad
s>0,\quad \epsilon>0.
\end{equation}

Using Proposition \ref{Sovp1} now, we can
 apply \eqref{diffresolv} to the Bernstein functions $\psi$
and $\varphi.$ It follows that
\[
\lim_{\epsilon\to
0}\,[(s+\psi(A_\epsilon))^{-1}-(s+\varphi(A_\epsilon))^{-1}]
=(s+\psi(A))^{-1}-(s+\varphi(A))^{-1},
\]
in the uniform operator topology. On the other hand,  by
(\ref{FEH}),
\[
|\lambda+\epsilon|\ge \cos(\beta/2)\,(|\lambda|+\epsilon),\qquad
\lambda\in \partial\Sigma_\beta,\qquad \epsilon>0.
\]
Therefore, if $\lambda\in
\partial\Sigma_\beta$ then
\[
\|(A-\lambda)^{-1}-(A-\lambda-\epsilon)^{-1}\|\le \epsilon
\frac{M^2(\pi-\beta, A)}{|\lambda(\lambda+\epsilon)|} \le
\frac{\epsilon M^2(\pi-\beta, A)}{\cos
(\beta/2)\,|\lambda|(|\lambda|+\epsilon)}.
\]
So, by (\ref{RRR}),  (\ref{need}) and the bounded convergence
theorem, we obtain that
\begin{eqnarray*}
\|r(A;s)-r(A_\epsilon;s)\| &\le& \frac{1}{2\pi} \int_{\partial
\Sigma_\beta}\,r(\lambda;z)
\|(A-\lambda)^{-1}-(A-\lambda-\epsilon)^{-1}\|\,|d\lambda|\\
&\le& \frac{\epsilon M^2(\pi-\beta, A)}{2\pi} \int_{\partial
\Sigma_\beta}\,\frac{|r(\lambda;s)|}{|\lambda|(|\lambda|+\epsilon)}
\,|d\lambda|\\
&\to& 0, \qquad \epsilon \to 0.
\end{eqnarray*}
Letting $\epsilon\to 0$ in (\ref{SA11}), (\ref{SA12}) follows.

Thus, $(\cdot + \psi(A))^{-1}$ satisfies \eqref{SA12} and extends holomorphically to $\Sigma_\omega$
as both $r(\cdot, A)$ and $(\cdot+\varphi(A))^{-1}$ have the latter property by Proposition \ref{Tr}
and Theorem
\ref{LB}, respectively.
Then, \cite[Appendix B, Proposition B5]{ABHN01} implies that $\Sigma_\omega \subset \rho (-\psi(A))$ and
the extension is given by   $(\cdot + \psi(A))^{-1}.$ This  yields (\ref{SA1}) for all $z \in \Sigma_\omega.$
\end{proof}

Now we are ready to prove the main results of this paper.
 under Bernstein functions.
The first statement of them shows that Bernstein functions leave the class of generators of sectorially bounded holomorpgic semigroups on a Banach space invariant and, moreover, preserve the holomorphy sectors.

\begin{thm}\label{Tangle}
Let $-A\in \mathcal{BH}(\theta)$ for some $\theta\in (0,\pi/2].$
Then for every $\psi\in
\mathcal{BF}$ one has $-\psi(A)\in\mathcal{BH}(\theta)$. Moreover,
if
\begin{equation}\label{Arb}
\alpha=\pi/2-\theta,  \quad 2< q < \pi/\alpha, \quad  \text{and} \quad
\pi/2<\gamma<(1-q^{-1})\pi,
\end{equation}
then
\begin{equation}\label{AESA}
\|z(z+\psi(A))^{-1}\|\le \widetilde{C}_{q,\gamma}(\theta),\qquad
z\in \Sigma_\gamma,
\end{equation}
where
\begin{equation}\label{defCC}
\widetilde{C}_{q,\gamma}(\theta)=
\frac{\tilde{M}_{\alpha,q}(A)\sin(\pi/q)}{\pi/q} \frac{\pi/q +\gamma}{\sin(\pi/q +\gamma)}+ \frac{4
M(A,\pi/2+\theta-\delta/2)}{\pi
\sin^2(\theta/2)\,\sin^2(\delta/4)},
\end{equation}
and $\delta=\pi/2+\theta-\gamma$.
\end{thm}

\begin{proof}
Let  $\varphi\in \mathcal{CBF}$ be the function associated with
$\psi\in \mathcal{BF}$. By Theorem \ref{LB},  if $-A\in
\mathcal{BH}(\theta)$ then $-\varphi(A)\in\mathcal{BH}(\theta)$.
Moreover, if $q$ and $\gamma$ satisfying \eqref{Arb} are fixed, then  by (\ref{SMS}),
\begin{equation}\label{SC10}
\|z r(A;z)\|\le  \frac{4 M(A,\pi/2+\theta-\delta/2)}{\pi
\sin^2(\theta/2)\,\sin^2(\delta/4)}, \qquad z\in \Sigma_\gamma,
\end{equation}
where $\delta=\pi/2+\theta-\gamma$.
Hence from Proposition \ref{SC1}, Theorem \ref{LB}  and (\ref{SC10}) it
follows that
(\ref{AESA}) holds with  the constant
$\widetilde{C}_{q,\gamma}(\theta)$ defined by (\ref{defCC}).
Since the choice of  $q$ and $\gamma$ satisfying (\ref{Arb}) is
arbitrary,
and $\gamma$ is arbitrarily close to $\pi-\alpha=\pi/2+\theta$ if $q$ is sufficiently close to $\pi/\alpha,$
we conclude that $-\psi(A)\in\mathcal{BH}(\theta)$.
\end{proof}

Theorem \ref{Tangle} has a version saying that Bernstein functions
preserve the class of  bounded (but not necessarily sectorially)
holomorphic $C_0$-semigroups. This version is an  immediate consequence
of Theorem \ref{Tangle} and the following lemma.

\begin{lemma}\label{sdvig}
Let $-A$ be the generator of a bounded $C_0$-semigroup on $X$ and let $\psi$ be a Bernstein function.
Suppose there exists $d\ge 0$ such that $-\psi(A+d)\in \mathcal{H}(\theta)$
for some $\theta \in(0,\pi/2].$
Then  $-\psi(A) \in H(\theta).$
\end{lemma}
\begin{proof}
By Proposition \ref{limit}, (i),
we have
\begin{equation}\label{sumphil1}
\psi(A+d)= \psi(A) + [\psi(\cdot +d)-\psi(\cdot)](A)=\psi(A) + B_d.
\end{equation}
By the product rule for the (extended) Hille-Phillips calculus we have
\[
(\psi(A+d) + s)^{-1}(B_d+s)^{-1}=(B_d+s)^{-1}(\psi(A+d) + s)^{-1}
\]
for sufficiently large $s>0.$
Then, by \cite[Section A-I.3.8, p. 24]{Nagel} (see also \cite[Theorem 1]{Aziz}) it follows that
the $C_0$-semigroups
$(e^{-t\psi(A+d)})_{t\ge 0}$ and $(e^{-tB_d})_{t\ge 0}$ commute.
Then, taking into account that $\dom(\psi(A))=\dom(\psi(A+d))$ by Proposition \ref{limit}, (ii)
and using \cite[Subsection II.2.7]{EngNag}, we conclude that
\begin{equation}\label{prodS}
e^{-t\psi(A)}=e^{-t\psi(A+d)}e^{tB_d},\qquad t>0.
\end{equation}
Since $(e^{tB_d})_{t \ge 0}$ extends to an entire function,  the statement of lemma follows.
\end{proof}

\begin{cor}\label{Tangle1}
 Let $-A$ be the generator of a bounded $C_0$-semigroup on $X$ such that
 $-A\in \mathcal{H}(\theta)$ for some $\theta\in (0,\pi/2].$
Then for every $\psi\in
\mathcal{BF}$ one has $-\psi(A)\in\mathcal{H}(\theta)$.
\end{cor}
\begin{proof}
Observe that if  $(e^{-tA})_{t\ge 0}$  is a bounded
$C_0$-semigroup  admitting a holomorphic extension to
 $\Sigma_\theta$, $\theta\in (0,\pi/2]$,
then by e.g. \cite[Proposition 3.7.2 $ b)$]{ABHN01}  we infer that for fixed $\theta'\in (0,\theta)$
and  big enough  $d>0$ the operator
$-(d+A)$ generates a $C_0$-semigroup $(e^{-t(d+A)})_{t\ge 0}$
which is holomorphic and sectorially bounded in $\Sigma_{\theta'}$.
Then, by  Theorem \ref{Tangle} the
$C_0$-semigroup
$(e^{-t\psi(d+A)})_{t\ge 0}$ is also holomorphic and
sectorially bounded in $\Sigma_{\theta'}$.
By Lemma \ref{sdvig}, $-\psi(A)$ generates a bounded $C_0$-semigroup which extends holomorphically to $\Sigma_{\theta'}.$
Since the choice of $\theta'\in (0,\theta)$ is arbitrary, the corollary follows.
\end{proof}

\begin{remark}\label{gapMir}
 It was claimed  in \cite{Mir2} that
if $-A$ is the generator of a bounded $C_0$-semigroup on $X$
then
$-A\in \cup_{\theta\in (0,\pi/2]}\mathcal{H}(\theta)$ implies the same property for $-\psi(A).$
Unfortunately, the proof of this fact in
\cite{Mir2} seems to contain a mistake. Specifically, in the
notation of \cite{Mir2}, the proof relies on the boundedness of
the operator $\psi(A)g_t(A)$ which not proved in \cite{Mir2}. (In
fact, it is easy to show that the boundedness of $\psi(A)g_t(A)$
is equivalent to the holomorphicity of $(e^{-t\psi(A)})_{t\ge 0}$).
Nonetheless, the holomorphicity of $(e^{-t\psi(A)})_{t\ge 0}$) was
proved in \cite{Mir1} for uniformly convex $X$ by means of
Kato-Pazy's criterion, see \cite{Ka70} and \cite[Corollaries 2.5.7
and 2.5.8]{Pazy}.
\end{remark}

Let $-A$ be the generator of a bounded $C_0$-semigroup $on$ with dense
range (and then trivial kernel by the mean ergodic theorem).
Consider so-called Stieltjes functions  $f:(0,\infty)\to
(0,\infty)$ which can be defined by the property that $1/f\in
\mathcal{CBF}.$ Recall also that for $f \in \mathcal {CBF}$ one
has that $1/f$ is Stieltjes so that the class of Stieltjes
functions is, in a sense, a reciprocal dual of the class of
complete Bernstein functions. Note that for Stieltjes $f$ the
operator $-f(A)$, does not, in general, generate a
$C_0$-semigroup. For example, if $f(z)=1/z$ then $f(A)=A^{-1}$,
and the corresponding counterexample can be found in \cite{GZT}.
On the other hand, for generators of sectorially bounded
holomorphic $C_0$-semigroups the situation is different and we
have the following statement, which follows from Theorem
\ref{Tangle}, (\ref{fpsi}) and the fact that inverses of
generators of bounded holomorphic $C_0$-semigroups of angle
$\theta$ generate semigroups of the same kind.  (For the latter
statement see e.g. \cite{Laub1}.)

Recall that $f \in \mathcal P$ if there exists a nonzero $\psi \in \mathcal{BF}$ such that $f=1/\psi.$
Note also that by the discussion preceding Proposition \ref{regSP}, for $f\in \mathcal P$ the operator $f(A)$ is well-defined in the (extended) holomorphic functional calculus
for $A \in {\rm Sect}(\theta), \theta \in (0,\pi/2],$ having dense range.
\begin{thm}\label{SpecCG1}
Suppose that $-A\in \mathcal{BH}(\theta)$ for some $\theta\in
(0,\pi/2]$ and $\ran(A)$ is dense. The for every $f\in
\mathcal{P}$ one has $-f(A)\in\mathcal{BH}(\theta)$.
\end{thm}

Now we can extend the classes of admissible $\psi$ and $f$ in  Theorems \ref{Tangle} and \ref{SpecCG1}.

\begin{cor}\label{SpecCG12}
Suppose that $-A\in \mathcal{BH}(\theta)$ for some $\theta\in
(0,\pi/2]$ and $\ran(A)$ is dense. If $h= \psi+f$, where $\psi\in
\mathcal{BF}$ and $f\in \mathcal{P},$ then
$-h(A)\in\mathcal{BH}(\theta)$.
\end{cor}
\begin{proof}
By Theorems \ref{Tangle} and \ref{SpecCG1}, we have
$-\psi(A)\in\mathcal{BH}(\theta),$
 $-f(A)\in\mathcal{BH}(\theta).$ By the product rule for the (extended) holomorphic functional calculus
 it follows that for every $s>0:$
 \begin{eqnarray*}
 [(s+\psi(\cdot))^{-1} (s+f(\cdot))^{-1}](A)&=&(s+\psi(A))^{-1} (s+f(A))^{-1}\\
 &=&(s+f(A))^{-1} (s+\psi(A))^{-1}.
\end{eqnarray*}
Hence, as in the proof of Lemma \ref{sdvig},  the semigroups
$(e^{-t\psi(A)})_{t\ge 0}$ and $(e^{-tf(A)})_{t\ge 0}$ commute. Then,
 by \cite[Subsection II.2.7]{EngNag},
$\overline{-\psi(A)-f(A)}$ generates a $C_0$-semigroup
$(e^{-t\psi(A)}e^{-tf(A)})_{t\ge 0}$, and therefore
$\overline{-\psi(A)-f(A)}\in \mathcal{BH}(\theta)$. From this, by
Proposition \ref{regSP}, it follows
$-h(A)\in\mathcal{BH}(\theta)$.
\end{proof}

Note that in the particular  case when $\varphi\in \mathcal{CBF}$
and $f$ is a  Stieltjes function (i.e. $1/f\in \mathcal{CBF}$), Corollary \ref{SpecCG1} was proved in
\cite[Theorem 6.4]{Laub}.

\section{Improving properties of Bernstein functions: Carasso-Kato functions}\label{Main2}

Let us first recall some notions and results from \cite{Carasso}.
To this aim and for formulating our results in this section the next definition
will be helpful.
\begin{defn}
A Bernstein function $\psi$ is said to be Carasso-Kato if
for every Banach space $X$, and every
bounded $C_0$-semigroup $(e^{-tA})_{t\ge 0}$ on $X,$ the $C_0$-semigroup
$(e^{-t\psi(A)})_{t \ge 0}$ is holomorphic.
\end{defn}
Following
\cite{Carasso}, denote the set of vaguely continuous convolution
semigroups of subprobability measures on $[0,\infty)$
by $\mathcal{T}$.
Let $\mathcal{I}$ stand for  the set of $(\mu_t)_{t\ge 0}\in
\mathcal{T}$ such that
a Bernstein function $\psi$  given by
$(\mu_t)_{t\ge 0}$ via
Bochner's formula \eqref{CMonG} is Carraso-Kato.
Let us finally denote by $\mathcal{T}_1\subset \mathcal{T}$ the set of
functions $[0,\infty) \ni t \mapsto \mu_t$ such that $\mu_t$
is continuously differentiable in $M_b([0,\infty))$ for $t>0$, with
\[
\|{\mu_t}'\|_{M_b}=\mbox{O}(t^{-1})\quad \mbox{as} \quad t\to 0+.
\]

Recall that by \cite[Theorem 4]{Carasso},
\begin{equation}\label{CK}
\mathcal{I}=\mathcal{T}_1.
\end{equation}
Moreover, $(\mu_t)_{t\ge 0}\in \mathcal{I}$ implies that $\psi\in
\mathcal{BF}$ defined  by  \eqref{CMonG} satisfies
\begin{equation}\label{run}
\psi(\C_+)\subset \overline{\Sigma}_\gamma-\beta:= \{\lambda\in
\C:\, \lambda+\beta\in \overline{\Sigma}_\gamma\}
\end{equation}
for some $\gamma\in (0,\pi/2)$ and $\beta\ge 0.$ Hence, as it was
shown in \cite{Carasso}, there exists $K>0$ such that
\[
|\psi(z)|\le K|z|^{2\gamma/\pi},\qquad |z|\ge 1,\qquad z\in \C_{+}.
\]
While \cite{Carasso} describes Carraso-Kato functions $\psi$
in terms of the families of measures $(\mu_t)_{t \ge 0}$ corresponding to $\psi$  via \eqref{CMonG}, the results of \cite{Carasso} are not  so easy to apply since one is usually given $\psi$ rather than  the corresponding family $(\mu_t)_{t \ge 0}.$ The aim of this section is to single out substantial classes of Carraso-Kato functions $\psi$ in terms of geometric properties of $\psi$ themselves.

Note first that
Corollary \ref{Tangle1}  yields immediately the following assertion.
\begin{cor}\label{GK}
Let  $\psi\in \mathcal{BF}$ and let $\varphi$ be a Carasso-Kato
function. Then  $\psi\circ \varphi$ is also Carasso-Kato.
\end{cor}
\begin{remark}
Let $\psi, \varphi \in \mathcal{BF},$ so that
\[
e^{-t\psi(z)}=\int_0^\infty e^{-zs}\,\mu_t(ds),\quad
e^{-t\varphi(z)}=\int_0^\infty e^{-zs}\,\nu_t(ds),\quad  z\ge
0,\quad t\ge 0,
\]
for some $(\mu_t)_{t \ge 0}, (\nu_t)_{t \ge 0} \in \mathcal T.$
Then, according to \cite[Theorem 5.19]{SchilSonVon2010},
\[
e^{-t(\psi\circ \varphi)(z)}= \int_0^\infty
e^{-z\tau}\,\eta_t(d\tau),
\]
where $(\eta_t)_{t\ge 0}\in \mathcal{T}$ is given by a convolution
formula
\begin{equation}\label{eta}
\eta_t(d\tau)=\int_0^\infty \nu_s(d\tau)\,\mu_t(ds).
\end{equation}
Thus, in the situation of Corollary \ref{GK}, the property
(\ref{CK}) implies that $(\eta_t)_{t\ge 0}$  belongs to
$\mathcal{T}_1$.
\end{remark}
\begin{example}\label{Ex}
$a)$\, It was proved in \cite{Carasso} that the function
\[
\varphi(z)=\log(z+1), \qquad z\ge 0,
\]
is  Carasso-Kato. If $-A$ is the generator of a bounded
$C_0$-semigroup, then as $\varphi \in \mathcal{CBF},$ Hirsch's
calculus yields
\[
\log(1+A)x=\int_1^\infty s^{-1}(s+A)^{-1}Ax\,ds,\qquad x\in
\dom(A).
\]
By the improving property of $\varphi,$  $-\log(1+A)$ is the
generator of a bounded holomorphic $C_0$-semigroup $(e^{-t\log(1+A)})_{t \ge
0}$ on $X$ given by
\[
e^{-t\log(1+A)}=(1+A)^{-t}=\int_0^\infty
e^{-sA}e^{-s}\frac{s^{t-1}}{\Gamma(t)}\,ds,\quad t\ge 0,
\]
so that $\varphi\in \mathcal{BF}$ corresponds to the semigroup
\[
\nu_t(ds)=\frac{s^{t-1}e^{-s}}{\Gamma(t)}\,ds, \qquad t>0,
\]
via \eqref{Bochner}. This was proved in \cite{Carasso} with a
quite complicated argument. From Corollary \ref{GK} we infer that for
every $(\mu_t)_{t\ge 0}\in \mathcal{T},$
\[
\eta_t(d\tau)= \left(\int_0^\infty
\frac{\tau^{s-1}e^{-s}}{\Gamma(s)}\,\mu_t(ds)\right)\,d\tau\in
\mathcal{T}_1.
\]

$b)$\, Consider a complete Bernstein function $\varphi(z)=\sqrt z.$ Observe that
\[
e^{-s\varphi(z)}=e^{-sz^{1/2}}=\frac{s}{2\sqrt{\pi}}\int_0^\infty
e^{-zr}\frac{e^{-s^2/4r}}{r^{3/2}}\,dr,
\]
and it easy to check that $\varphi$ is Carraso-Kato, \cite{Carasso}.
By  Corollary \ref{GK}, for each $(\mu_t)_{t\ge 0}\in
\mathcal{T},$
\[
\eta_t(dr)=\frac{1}{2\sqrt{\pi}}\left(\int_0^\infty s
e^{-s^2/4r}\mu_t(ds)\right)\,\frac{dr}{r^{3/2}}\in \mathcal{T}_1.
\]
\end{example}

We proceed with several new conditions for a
function to be Carraso-Kato. Roughly, they say that the function is Carasso-Kato if it  shrinks a large angular
sector to a smaller one.

The first statement provides a geometric condition for a stronger version of the
Carraso-Kato property.
\begin{thm}\label{IlCBF}
Let $\psi$ be  a Bernstein function. Suppose there exist
$\theta_1 \in (0,\pi)$ and $\theta_2  \in (\pi/2, \pi)$ such
that $\psi$ admits a  continuous extension $\tilde{\psi}$ to
$\overline{\Sigma}_{\theta_2}$  which is  holomorphic in
$\Sigma_{\theta_2},$ and
\begin{equation}\label{MCon4}
\tilde{\psi}(\overline{\Sigma}_{\theta_2}^{+})\subset
\overline{\Sigma}_{\theta_1}^{+}.
\end{equation}
Then the following holds.
\begin{itemize}
\item [(i)] One has $\tilde{\psi}\in \mathcal{B}(\Sigma_{\theta_2})$ so that
 if
$A\in \Sect(\gamma)$, $\gamma<\theta_2$, is injective then $\tilde{\psi}(A)$ is well-defined in the (extended) holomorphic functional
calculus.
\item [(ii)] One has
\[
\tilde{\psi}(A)\in \Sect(\omega), \qquad
\omega=\gamma\cdot\frac{\theta_1}{\theta_2}.
\]
In particular, if $0 \gamma\in (0,\frac{\pi\theta_2}{2\theta_1})$ then
\begin{equation}\label{BHH0}
-\tilde{\psi}(A)\in \mathcal{BH}(\theta),\qquad
\theta=\frac{\pi}{2}\left(1-\frac{2\gamma}{\pi}\cdot\frac{\theta_1}{\theta_2}\right).
\end{equation}
\end{itemize}
\end{thm}

\begin{proof}
Let
\[
\alpha:=\frac{\theta_2}{\pi}\in (1/2,1),\qquad
\beta:=\frac{\pi}{\theta_1}>1.
\]
Then, by  (\ref{MCon4}), both functions
\[
\tilde{\psi}_\alpha(\lambda):=\tilde{\psi}(\lambda^\alpha),\quad
\varphi(\lambda):=[\tilde{\psi}_\alpha(\lambda)]^\beta,\quad \lambda\in \C\setminus (-\infty,0),
\]
map the upper half-plane $H^{+}$ into itself. Hence, using Theorem
\ref{Shill}, (iii), we conclude that
\begin{equation}\label{Incc}
\tilde{\psi}_\alpha \in \mathcal{CBF},\qquad \varphi\in
\mathcal{CBF}.
\end{equation}
In particular, from the first inclusion in (\ref{Incc}) and Lemma
\ref{psi}, (ii),
it follows that  $\tilde{\psi}\in
\mathcal{B}(\Sigma_{\theta_2}).$ Therefore, as $A\in \Sect(\gamma)$ with $\gamma<\theta_2$, the operator $\tilde{\psi}(A)$ is
well-defined in the (extended) holomorphic functional calculus, and (i) is proved.

Since  $\gamma/\alpha<\pi$,
 Proposition \ref{frpow} implies that
\begin{equation}\label{sectt}
A^{1/\alpha}\in \Sect(\gamma/\alpha).
\end{equation}

Moreover, if $\gamma_0 \in (\gamma,\pi)$ is fixed and $\gamma'$ satisfies
$\gamma<\gamma_0<\gamma'<\pi\alpha$ then using Proposition
\ref{PrBBL10} it follows that
\[
M(A^{1/\alpha},\gamma'/\alpha)= \sup_{z\in
\Sigma_{\pi-\gamma'/\alpha}}\, \|z(z+A^{1/\alpha})^{-1}\|\le
C M(A,\gamma_0),
\]
where the constant $C$ does not depend on $A.$
Recall that the Hirsch functional calculus and the (extended) holomorphic functional calculus are compatible by Proposition \ref{Sovp2}.
Thus, since $\tilde{\psi}$ is complete Bernstein,  by (\ref{AB2}), Theorem \ref{LB} and
(\ref{sectt}),
\begin{equation}\label{this}
\tilde{\psi}_\alpha(A^{1/\alpha})
=[\varphi(A^{1/\alpha})]^{1/\beta}\in \Sect(\gamma/(\alpha
\beta)),
\end{equation}
where the operators are defined by the Hirsch functional calculus.

Furthermore, if $\omega_0 \in (\omega, \pi)$ is fixed and  $\omega'$ satisfies $\omega<\omega_0<\omega'<\pi,$
then Theorem \ref{LB} implies that
\[
\|z(z+\tilde{\psi}_\alpha(A^{1/\alpha}))^{-1}\|\le C
M(A^{1/\alpha},\omega_0),\quad z\in \Sigma_{\pi-\omega'},
\]
where  once again $C$ does not depend on $A.$

Next, as $A$ is injective, we use  Propositions \ref{Sovp2} and
\ref{frpow3} to infer  from (\ref{this})  that
\[
\tilde{\psi}(A)=\tilde{\psi}_\alpha(A^{1/\alpha})\in
\Sect(\omega),\quad
\omega=\frac{\gamma}{\alpha\beta}=\gamma\cdot\frac{\theta_1}{\theta_2},
\]
where $\tilde{\psi}(A)$ is defined by the (extended) holomorphic functional
calculus. In particular, if $\gamma \in (0, \frac{\pi\theta_2}{2\theta_1}$, then
\[
-\tilde{\psi}(A)\in \mathcal{BH}(\theta), \qquad
\theta=\frac{\pi}{2}-\omega=\frac{\pi}{2}
\left(1-\frac{2\gamma}{\pi}\cdot\frac{\theta_1}{\theta_2}\right).
\]
\end{proof}

\begin{cor}\label{IlCBFcor}
Let $\psi$ be  a Bernstein function satisfying the conditions of
Theorem \ref{IlCBF} where, in addition,  $\theta_1<\theta_2$. Let $-A$ be the generator of a bounded
$C_0$-semigroup on $X$. Then
\begin{equation}\label{BHH}
-\psi(A)\in \mathcal{BH}(\theta), \qquad
\theta=\frac{\pi}{2}\left(1-\frac{\theta_1}{\theta_2}\right),
\end{equation}
where $\psi(A)$ is given by  the (extended) Hille-Phillips
functional calculus. Moreover, if $-A\in \mathcal{BH}(\theta_0)$,
$\theta_0\in (0,\pi/2]$, then
\begin{equation}\label{BHHT}
-\psi(A)\in \mathcal{BH}(\theta), \qquad \theta=
\theta_0+\left(\frac{\pi}{2}-\theta_0\right)\left(1-\frac{\theta_1}{\theta_2}\right).
\end{equation}
\end{cor}

\begin{proof}
If operator $A$ is injective, then (\ref{BHH}) follows from
Theorem \ref{IlCBF} with $\gamma=\pi/2$ and Proposition
\ref{Sovp1}.

Assume now that $A$ is not injective. Then consider the injective
operators $A_\epsilon$ defined by
\[
A_\epsilon:=A+\epsilon,\quad \epsilon>0.
\]
Using once again  Theorem \ref{IlCBF} with  $\gamma=\pi/2$ and
Proposition \ref{Sovp1} we obtain that
\begin{equation}\label{epsP}
-\psi(A+\epsilon)\in \mathcal{BH}(\theta), \quad
\theta=\frac{\pi}{2}\left(1-\frac{\theta_1}{\theta_2}\right).
\end{equation}

Moreover, since for any $\omega'\in (\pi/2,\pi),$
\[
\sup_{z\in \Sigma_{\pi-\omega'}}\,\|z(z+A+\epsilon)^{-1}\|\le \sup_{z\in
\Sigma_{\pi-\omega'}}\,\|z(z+A)^{-1}\|=M(A,\omega'),\quad \epsilon>0,
\]
 the proof of  Theorem \ref{IlCBF} implies that
if $\omega=\pi/2-\theta$ and $ \omega_0\in (\omega,\pi)$  then
\[
\sup_{\epsilon\in (0,1)}\,\sup_{z\in \Sigma_{\pi-\omega_0}}\,
\|z(z+\psi(A_\epsilon))^{-1}\|<\infty.
\]
From this, the resolvent convergence given by (\ref{Psi4}), and
Vitali's theorem it follows that
\[
\sup_{z\in \Sigma_{\pi-\omega_0}}\,
\|z(z+\psi(A))^{-1}\|<\infty,
\]
for all $\omega_0\in (\omega,\pi).$ In other words,  $-\psi(A)\in \mathcal{BH}(\theta)$.

Finally, if $-A\in \mathcal{BH}(\theta_0)$, then $A\in
\Sect(\pi/2-\theta_0)$ and, similarly, using (\ref{BHH0}) with
 $\gamma=\pi/2-\theta_0$, we obtain  (\ref{BHHT}).
\end{proof}

In a manner similar to the proof of Corollary \ref{IlCBFcor},
Theorem \ref{IlCBF} yields also the following assertion providing a geometric
condition for the Carasso-Kato property.
The statement is, in fact, Theorem \ref{BFIntro1} mentioned in Introduction.

\begin{cor}\label{IlCBF1}
Let $\psi$ be a Bernstein function. Suppose there exist $\theta\in
(\pi/2,\pi)$ and $r>0$ such that  $\psi$ admits a continuous
extension $\overline{\Sigma}_\theta$ which is holomorphic in  $\Sigma_\theta$, and
\begin{equation}\label{CKr}
\psi(\lambda)\in \overline{\Sigma}_{\pi/2}^{+}\quad
\mbox{for}\quad \lambda\in \overline\Sigma_\theta^{+},\quad |\lambda|\ge r.
\end{equation}
Then $\psi$ is Carasso-Kato function. Moreover,
for any generator
$A$ of a bounded  $C_0$-semigroup on $X$,
one has $-\psi(A) \in \mathcal H(\frac{\pi}{2}(1-\frac{\pi}{2\theta})).$
\end{cor}

\begin{proof}
From (\ref{CKr}) it follows  that there exists  $d>0$ such
that function $\psi_d$ given by
\[
\psi_d(\lambda):=\psi(\lambda+d),\quad \lambda \in
\overline{\Sigma}_\theta,
\]
is holomorphic in $\Sigma_\theta$ and continuous in $\overline{\Sigma}_\theta$,  and it satisfies
\begin{equation}\label{CKr1}
\psi_d(\lambda)\in \overline{\Sigma}_{\pi/2}^{(+)}\quad
\mbox{for}\quad \lambda\in \overline{\Sigma}_\theta^{(+)}.
\end{equation}
 Therefore, by Corollary \ref{IlCBFcor}
 with $\theta_2=\theta$ and $\theta_1=\pi/2$, if
$-A$ is the generator of a bounded $C_0$-semigroup then
$-\psi_d(A)=-\psi(A+d)\in \mathcal{BH}(\theta_0),$ where
$$
\theta_0=\frac{\pi}{2}\left(1-\frac{\pi}{2\theta}\right) \in
\left(0,\frac{\pi}{2}\right).
$$
Then, by Lemma \ref{sdvig},
we conclude that $-\psi(A)$ generates a
bounded $C_0$-semigroup possessing holomorphic extension to
$\Sigma_{\theta_0}$.
\end{proof}

Let us recall now the next  result by  Fujita \cite[p. 337 and Lemma 2]{F1}.

\begin{thm}\label{F1}
Let $\alpha\in (0,1)$, $\theta_\alpha=\pi/(1+\alpha)$ and
$\Theta\in (\theta_\alpha,\pi)$ Let $\psi\in \mathcal{BF}$ satisfy
the following conditions:
\begin{itemize}
\item [(A1)]  $\psi$ admits a continuous extension
to $\overline{\Sigma}_{\Theta}$ which is holomorphic in $\Sigma_{\Theta}$.

\item [(A2)] one has
\begin{equation}\label{A2}
\lim_{r\to\infty}\,\frac{\psi(re^{i\theta})}{\psi(r)}=e^{i\alpha\theta}\quad
\mbox{uniformly in}\quad |\theta|\le \theta_0,
\end{equation}
and $\psi$ is regularly varying of order $\alpha$,

\item [(A3)] $\psi(re^{\pm i\theta_\alpha})/r$ are integrable in a right neighborhood of $0$.
\end{itemize}
Then $\psi$ is a Carasso-Kato function.
\end{thm}

Note that Theorem \ref{F1} follows from Corollary \ref{IlCBF1}.
Moreover, Corollary \ref{IlCBF1} shows that one can omit the
condition $(A3)$ and the assumption that $\psi$
is a regularly varying. Indeed  from (A1) ,  (A2) and the properties
\[
\psi(r)>0,\quad r>0,\quad
\alpha\Theta<\alpha\theta_{\alpha}=\frac{\pi\alpha}{1+\alpha}<\pi/2,
\]
it follows that for a large $r>0,$
\[
\psi(\lambda)\in
\overline{\Sigma}_{\alpha\Theta}^{(+)}\subset\overline{\Sigma}_{\pi/2}^{+}\quad
\mbox{for}\quad \lambda\in \overline{\Sigma}_\Theta^{+}(r),
\]
hence (\ref{CKr}) holds.

Now we turn our attention to Carraso-Kato functions $\psi$ which are, in addition,  \emph{complete}
Bernstein functions. As in the situation of Theorem \ref{IlCBF}, we first require $\psi$ to map generators of bounded semigroups
into the generators of \emph{sectorially bounded} holomorphic semigroups.
Such $\psi$ can, in fact, be
\emph{characterized} in an elegant way as the following statement shows. (It corresponds to Theorem \ref{CBFIntro}
from Introduction.)

\begin{thm}\label{IlCBF11}
Let $\psi$ be  a complete Bernstein function and let $\gamma \in
(0,\pi/2)$ be fixed. The the following assertions are equivalent.
\begin{itemize}

\item [(i)]
One has
$$ \psi(\overline{\mathbb C}_+) \subset
 \overline{\Sigma}_\gamma.$$

\item [(ii)]  For each (complex) Banach space $X$ and each generator $-A$ of a bounded
$C_0$-semigroup on $X,$ the operator $-\psi(A)$ belongs to $\mathcal{BH}(\pi/2-\gamma)$.
\end{itemize}
\end{thm}

\begin{proof}
The implication (ii) $\Rightarrow$ (i) follows from \cite[Theorem
4]{Carasso} and its proof.
So, it suffices to prove that (i) implies (ii).

Assume that (ii) is true. Then, by  Corollary \ref{IlCBFcor} and
Proposition \ref{psi1}, we obtain that
\begin{equation}\label{LimS}
-\psi(A)\in \mathcal{BH}(\theta), \quad
\theta=\frac{\pi}{2}\left(1-\frac{\tilde{\theta}_0}{\theta_0}\right),
\end{equation}
where $\theta_0\in (\pi/2,\pi)$ is such that
\[
|\cos\theta_0|=\frac{\cot\gamma}{\cot\gamma+1},
\]
and $\tilde{\theta_0}=\tilde{\theta_0}(\theta_0) \in (0,\pi/2),$ is defined by the equation
\[
\cot \tilde{\theta}_0=  C(\theta_0),\qquad
C(\theta):=\frac{\cot\gamma+1}{\sin\theta}\left(\frac{\cot\gamma}{\cot\gamma+1}-|\cos\theta|\right).
\]
Note that
\[
\lim_{\theta_0\to\pi/2}\, C(\theta_0)=\cot \gamma \quad
\]
and therefore
\[
 \lim_{\theta_0\to\pi/2}\,\tilde{\theta}_0(\theta_0)=\gamma.
\]
Thus considering  $\theta_0$ in (\ref{LimS}) arbitrarily close
to $\pi/2$, we obtain the assertion (ii).
\end{proof}

Let us recall that (\ref{run}) is necessary for  $\psi\in
\mathcal{BF}$ to be a Carasso-Kato function. The next
statement show that if moreover $\psi\in \mathcal{CBF}$ then
(\ref{run}) is also sufficient thus providing a characterization of the Carasso-Kato property
for complete Bernstein functions.

\begin{cor}\label{IlCBF12}
Let $\psi$ be  complete Bernstein function. Then $\psi$ is
Carasso-Kato  if and only if there exist $\gamma\in (0,\pi/2)$
and $\beta\ge 0$ such that (\ref{run}) holds. Moreover, if
(\ref{run}) holds and if $-A$ generates  a bounded
$C_0$-semigroup on $X,$ then $-\psi(A) \in \mathcal H(\pi/2-\gamma).$
\end{cor}

\begin{proof}
It is sufficient to show  that (\ref{run}) implies that $\psi$ is
Carasso-Kato. If (\ref{run}) is satisfied,  by  applying Theorem
\ref{IlCBF11} to
 $\psi_\beta\in \mathcal{CBF}$, $\psi_\beta(\lambda):=\psi(\lambda+\beta)$,
we obtain that
\[
-\psi_\beta(A)=-\psi(\beta+A)\in \mathcal{BH}(\theta),\quad
\theta=\pi/2-\gamma.
\]
for any generator $-A$ of a bounded $C_0$-semigroup on $X$. Then, using Lemma \ref{sdvig},
we conclude that $-\psi(A)$ generates a bounded
$C_0$-semigroup having a holomorphic extension to
$\Sigma_{\theta}$.
\end{proof}

\begin{rem}\label{log}
Observe that
$\psi(z)=\log(1+z)$ is complete Bernstein and, moreover,
\[
|{\rm Im}(\psi(z))|\le \pi/2,\qquad z\in \C_{+}.
\]
Thus,
for every  $\gamma\in (0,\pi/2)$ there is  $\beta>0$ such that  $\psi$ satisfies (\ref{run}).
Then, by
Corollary \ref{IlCBF12}, we obtain
that $\psi$ is
Carasso-Kato, and, moreover, for any generator $-A$ of a bounded
$C_0$-semigroup  the semigroup generated by $-\log(A+I)$
admits a holomorphic extension to  $\C_{+}=\Sigma_{\pi/2}$. This fact complements
Example \ref{Ex}, a).
\end{rem}

\section{Appendix}
It is an open question whether for any $\alpha\in
(0,1),$
\begin{equation}\label{Stat1}
\psi \in \mathcal{BF} \Longrightarrow [\psi(\lambda^\alpha)]^{1/\alpha}\in \mathcal{BF}.
\end{equation}
A positive answer to this question would allow one to apply the
methods from \cite{Laub} directly and to obtain the results from
Section \ref{Main1} in a comparatively simple way. Let us analyse
the property \eqref{Stat1} in some more details.

Apart from the situation described in (\ref{Stat}), it is known
that (\ref{Stat1} is true if
$\alpha=1/n, n \in \N$ \cite[Proposition 7.1]{Laub} (see, also
\cite[Remark 12]{PS}). The following Lemma \ref{genG} generalizes
\cite[Proposition 7.1]{Laub} and (\ref{Stat}) in the case when
$\alpha\in (0,1/2]$ and $\psi\in \mathcal{BF}$ and it extends these statements  for any
$\alpha\in (0,1)$ if $\psi$ is a so-called special Bernstein function. Recall that a non-zero Bernstein function $\psi$ is said to be a special Bernstein function, if $z/\psi(z)$ is a Bernstein function.

\begin{prop}\label{PrelP}
Let $\psi$ be a  Bernstein function and let  $\alpha\in (0,1)$.
For $\beta>0$ define
\[
\tilde{\psi}_{\alpha,\beta}(z):=
\left(\frac{\psi(z^{\alpha})}{z^\alpha}\right)^\beta, \qquad z
>0.
\]
Then $ \tilde{\psi}_{\alpha,\beta}$ is complete monotone for all
$\alpha\in (0,1/2]$ and $\beta>0.$ If $\psi$ is a special
Bernstein function  then $ \tilde{\psi}_{\alpha,\beta}$ is
complete monotone for all  $\alpha\in (0,1)$ and $\beta>0.$
\end{prop}

\begin{proof}
If $\psi\in \mathcal{BF}$ and $\alpha\in (0,1/2]$, then by
\cite[Proposition 3.6]{Sonia} one has
\[
f_\alpha(z):=z^{1-\alpha}\psi(z^\alpha)\in \mathcal{CBF},
\]
so that $z/f_\alpha(z)\in \mathcal{CBF}$ and, by \cite[Theorem
3.6, (ii)]{SchilSonVon2010}, for any $\beta>0$,
\[
\left[\frac{\psi(z^\alpha)}{z^\alpha}\right]^\beta=
\left[\frac{z^{1-\alpha}\psi(z^\alpha)}{z}\right]^\beta=
\left[\frac{z}{f_\alpha(z)}\right]^{-\beta}
\]
is complete monotone. Let now $\alpha\in (0,1)$ and $\psi$ be a
special Bernstein function so that $\psi(z)=z/f(z)$, $f\in
\mathcal{BF}$. Then $f(z^\alpha)$ is Bernstein, and by
\cite[Theorem 3.6, (ii)]{SchilSonVon2010} the function
$\tilde{\psi}_{\alpha,\beta}$ given by
\[
\tilde{\psi}_{\alpha,\beta}(z):= [f(z^{\alpha})]^{-\beta}, \qquad
z >0,
\]
is complete monotone for any $\beta>0$.
\end{proof}

\begin{remark}\label{Arem}
 Let us note two partial cases of  Proposition \ref{PrelP}. Let
$\beta=1/\alpha-1$ and
\begin{equation}\label{FunAB}
\tilde{\psi}_{\alpha}(z):=
\left(\frac{\psi(z^\alpha)}{z^\alpha}\right)^{1/\alpha-1},\qquad
\alpha\in (0,1).
\end{equation}
If $\psi\in \mathcal{BF}$ and $\alpha\in (0,1/2]$ then
$\tilde{\psi}_{\alpha}$  is completely monotone.
If  $\psi\in \mathcal{SBF}$ and  $\alpha\in (0,1]$ then
$\tilde{\psi}_{\alpha}$ is completely monotone as well.
\end{remark}

\begin{lemma}\label{genG}
Let $\psi$ be a  Bernstein function and let
\[
\psi_\alpha(z):= [\psi(z^\alpha)]^{1/\alpha},\qquad \alpha\in
(0,1).
\]
If $\alpha\in (0,1/2],$ then
$
\psi_\alpha(z)\in  \mathcal{BF}$. If $\alpha\in (0,1)$ and in addition $\psi$ is a special Bernstein function, then
$
\psi_\alpha$ is a special Bernstein function too.
\end{lemma}

\begin{proof}
We have
\begin{equation}\label{GHT}
\psi_\alpha'(z)=\psi'(z^\alpha)\tilde{\psi}_\alpha(z),\qquad z>0,
\end{equation}
where  $\tilde{\psi}_\alpha$ is defined by (\ref{FunAB}). Then,
since $\psi'(z^\alpha)$ is complete monotone (as composition of
Bernstein and complete monotone functions, see \cite[Theorem 3.7]{SchilSonVon2010}),
$\psi_\alpha\in
\mathcal{BF}$ by Remark \ref{Arem}.

If $\psi\in \mathcal{SBF}$ and $\alpha\in (0,1)$ then  $\psi=z/f$ for some  $f\in
\mathcal{SBF}.$ Hence
$
f_\alpha(z):=[f(z^{1/\alpha})]^\alpha\in \mathcal{BF},
$
and then
$
\psi_\alpha(z)=\frac{z}{f_\alpha(z)}\in \mathcal{SBF}.
$
\end{proof}

\begin{example}\label{SpecF}
Note that (\ref{Stat1}) does not imply that $\psi$ belongs to
$\mathcal{SBF}$. For instance,
\[
\psi(z):=1-\frac{1}{(1+z)^2}=\frac{z(2+z)}{(1+z)^2},\qquad z>0,
\]
is Bernstein, while $\psi \not \in \mathcal{SBF}.$ Indeed,
\[
\frac{z}{\psi(z)}=\frac{(1+z)^2}{2+z}=z+\frac{1}{2+z}\not\in\mathcal{BF}.
\]
On the other hand, we have
\[
\psi(z)=\frac{z}{f_1(z)f_2(z)},
\]
where
\[
f_1(z)=1+z\in \mathcal{CBF}, \qquad \text{and} \qquad
f_2(z)=\frac{1+z}{2+z}=1-\frac{1}{2+z}\in \mathcal{BF}.
\]
Thus, for $\alpha\in (0,1),$
\[
([\psi(z^\alpha)]^{1/\alpha})' =\frac{2}{(1+z^\alpha)^3} \cdot
[f_1(z^\alpha)]^{-(1/\alpha-1)}\cdot
[f_2(z^\alpha)]^{-(1/\alpha-1)}
\]
is complete monotone, so that $[\psi(z^\alpha)]^{1/\alpha} \in
\mathcal{BF}$ for any $\alpha\in (0,1)$.
\end{example}
\section{Acknowledgment}

We are grateful to C. J. K. Batty for reading parts of this manuscript.

\end{document}